\documentclass[11pt,reqno]{amsart}

\usepackage[utf8]{inputenc}
\usepackage[english]{babel}

\usepackage{graphicx}
\usepackage{epstopdf}
\usepackage{enumerate}
\usepackage{enumitem}
\usepackage{amsthm}
\usepackage{amsmath}
\allowdisplaybreaks[1] %% That's for 'align' environment to split into two pages when needed
\usepackage{amsfonts}
\usepackage{mathrsfs}
\usepackage{subcaption}
\usepackage{eurosym}
\usepackage{dsfont}
\usepackage[colorlinks=true, citecolor=blue, linkcolor=blue]{hyperref}
\usepackage{chngpage}
\usepackage{float}
\usepackage{listings}
\usepackage{color}
\usepackage{tikz}

\definecolor{darkgreen}{rgb}{0.0, 0.2, 0.13}
\definecolor{darkpastelgreen}{rgb}{0.01, 0.75, 0.24}
\usepackage{array}
\usepackage{booktabs}
\DeclareMathOperator*{\argmax}{arg\,max}
\DeclareMathOperator*{\sgn}{sgn}
\DeclareMathOperator*{\erf}{erf}
\DeclareMathOperator*{\erfc}{erfc}

\DeclareMathOperator*{\cov}{\mathbb{C}ov}
\DeclareMathOperator*{\var}{\mathbb{V}ar}

\newtheorem{proposition}{Proposition}

\newtheorem{lemma}{Lemma}

\newtheorem{corollary}{Corollary}

\newcommand{\R}{\mathbb{R}}
\newcommand{\N}{\mathbb{N}}

\newcommand{\p}{\mathbb{P}}
\newcommand{\e}{\mathbb{E}}

\newcommand{\ep}{\varepsilon}
\newcommand{\eqd}{\overset{\rm d}{=}}

\newcommand{\vk}[1]{\boldsymbol{#1}}

\newcommand{\ap}{{c_+}}
\newcommand{\am}{{c_-}}
\newcommand{\apm}{{\vk c}}
\newcommand{\kummer}{{}_1F_1}

\renewcommand{\hat}{\widehat}

\makeatletter
\def\env@dmatrix{\hskip -\arraycolsep
  \let\@ifnextchar\new@ifnextchar
  \extrarowheight=2ex
  \array{*\c@MaxMatrixCols{>{\displaystyle}c}}}

\makeatother

\allowdisplaybreaks

\renewcommand{\tilde}{\widetilde}

\date{\today}

\makeatletter
\@namedef{subjclassname@2020}{%
  \textup{2020} Mathematics Subject Classification}
\makeatother

\begin{document}

\title[Lower bound for the expected supremum of fBm]{Lower bound for the expected supremum of fractional Brownian motion using coupling}

\author[K.\ Bisewski]{Krzysztof Bisewski}
\address{Department of Actuarial Science, University of Lausanne, UNIL-Dorigny, 1015 Lausanne, Switzerland}
\email{Krzysztof.Bisewski@unil.ch}


\begin{abstract}
We derive a new theoretical lower bound for the expected supremum of drifted fractional Brownian motion with Hurst index $H\in(0,1)$ over (in)finite time horizon. Extensive simulation experiments indicate that our lower bound outperforms the Monte Carlo estimates based on very dense grids for $H\in(0,\tfrac{1}{2})$. Additionally, we derive the Paley-Wiener-Zygmund representation of a Linear Fractional Brownian motion and give an explicit expression for the derivative of the expected supremum at $H=\tfrac{1}{2}$ in the sense of recent work by Bisewski, D{\c e}bicki \& Rolski (2021).
\end{abstract}

\maketitle

%---------------------------------------------------------
%---------------------------------------------------------
%---------------------------------------------------------

\section{Introduction}

Let $\{B_H(t),t\in\R_+\}$, where $\R_+:=[0,\infty)$ be fractional Brownian motion with Hurst index $H\in(0,1)$ (or $H$-fBm), that is, a centred Gaussian process with the covariance function
\begin{equation}\label{def:fbm_covariance}
\cov(B_H(t),B_H(s)) = \frac{1}{2}(s^{2H}+t^{2H}-|t-s|^{2H}), \quad s,t\in\R_+.
\end{equation}
In this manuscript we consider the expected supremum of fractional Brownian motion with drift $a\in\R$ over time horizon $T>0$, that is
\begin{align*}
\mathscr M_H(T,a) := \e\Big(\sup_{t\in[0,T]} B_H(t) - at\Big).
\end{align*}
Even though the quantity $\mathscr M_H(T,a)$ is so fundamental in the theory of extremes of fractional Brownian motion, its value is known explicitly only in two cases: $H=\tfrac{1}{2}$ and $H=1$, when $B_H$ is a standard Brownian motion and a straight line with normally distributed slope, respectively. For general $H\in(0,1)$, the value of $\mathscr M_H(T,a)$ could, in principle, be approximated using Monte Carlo methods by simulation of an fBm on a dense grid, i.e.
\begin{align*}
\mathscr M^n_H(T,a) := \e\left(\sup_{t\in\mathcal T_n} \{B_H(t) - at\}\right),
\end{align*}
where $\mathcal T_n := \{0,\tfrac{T}{n},\tfrac{2T}{n},\ldots,T\}$. However, this approach can lead to substantial errors. In \cite[Theorem~3.1]{borovkov2017bounds} it was proven that the absolute error $\mathscr M_H(T,0)-\mathscr M^n_H(T,0)$ behaves roughly (up to logarithmic terms) like $n^{-H}$, as $n\to\infty$. This becomes problematic, as $H\downarrow0$, when additionally $\mathscr M_H(T,0)\to\infty$; see also \cite{MR3519723}. Similarly, the error is expected to be large, when $T$ is large, since in that case more and more points are needed to cover the interval $[0,T]$. Surprisingly, even as $H\uparrow 1$ one may also encounter problems because then $M_H(\infty,a) \to\infty$ for all $a>0$, see \cite{MR4018824, BDM21}. Since the estimation of $\mathscr M_H(T,a)$ is so challenging, many works are dedicated to finding its theoretical upper and lower bounds. The most up-to-date bounds for $\mathscr M_H(T,0)$ can be found in \cite{borovkov2017bounds, borovkov2018new}, see also \cite{Sha96, vardar2015bounds} for older results. The most up-to-date bounds for $\mathscr M_H(\infty,a)$ can be found in \cite{BDM21}.

In this work we present a new theoretical lower bound for $\mathscr M_H(T,a)$ for general $T>0$, $a\in\R$ (including the case $T=\infty$, $a>0$). Our approach is loosely based on recent work \cite{bisewski2021derivatives}, where the authors consider a coupling between $H$-fBms with different values of $H\in(0,1)$ on the Mandelbrot \& van Ness field \cite{mandelbrot1968fractional}. The idea of considering such a coupling dates back at least to \cite{peltier1995multifractional, benassi1997elliptic}, who introduced a so-called multi-fractional Brownian motion. In this manuscript, we consider a coupling provided by the family of Linear Fractional Stable motions with $\alpha=2$, see \cite[Chapter~7.4]{samorodnitsky1994stable} which we will call the linear fractional Brownian motion. Conceptually, our bound for $\mathscr M_H(T,a)$ is very simple --- it is defined as the expected value of the $H$-fBm at the time of maximum of the corresponding $\tfrac{1}{2}$-fBm (i.e. Brownian motion). The difficult part is the actual calculation of this expected value. This is described in detail in Section \ref{s:main_results}. Our new lower bound, which we denote $\underline{\mathscr M}_H(T,a)$, is introduced in Corollary~\ref{coro:MH}.

Our numerical experiments show that $\underline{\mathscr M}_H(T,a)$ performs exceptionally well in the subdiffusive regime $H\in(0,\tfrac{1}{2})$. In fact, the numerical simulations indicate that $\underline{\mathscr M}_H(T,a)$ gives a \emph{better approximation} to the ground truth that the Monte Carlo estimates with as many as $2^{16}$ gridpoints, i.e.
\begin{equation*}
\underline{\mathscr M}_H(T,a) \geq \mathscr M^n_H(T,a), \quad n=2^{16}
\end{equation*}
for all $H\in(0,\tfrac{1}{2})$. We emphasize that $\underline{\mathscr M}_H(T,a)$ is the theoretical \emph{lower} bound for $\mathscr M_H(T,a)$, which makes the result above even more surprising.

The manuscript is organized as follows. In Section~\ref{s:preliminaries} we define the Linear Fractional Brownian motion and establish its Paley-Wiener-Zygmund representation. We also recall the formula for the joint density of the supremum of drifted Brownian motion and its time and introduce a certain functional of the 3-dimensional Bessel bridge, which plays an important role in this manuscript. In Section \ref{s:main_results} we show our main results; our lower bound $\underline{\mathscr M}_H(T,a)$ is presented in Corollary \ref{coro:MH}. Additionally, in Corollary \ref{coro:derivative} we present an explicit formula for the derivative $\frac{\partial}{\partial H} \mathscr M_{H}(T,a)\vert_{H=1/2}$, which was given in terms of definite integral in \cite{bisewski2021derivatives}. The main results are compared to numerical simulations in Section \ref{s:numerical_experiments}, where the results are also discussed. The proofs of main results are given in Section \ref{s:proofs}. In Appendix~\ref{appendix:special_functions} we recall the definition and properties of confluent hypergeometric functions. Finally, in Appendix~\ref{appendix:calculations} we enclosed various calculations needed in the proofs.

%----------------------------------------------------
%----------------------------------------------------
%----------------------------------------------------

\section{Preliminaries}\label{s:preliminaries}

\subsection{Linear Fractional Brownian motion}

In this section we introduce the definition of the Linear Fractional Brownian motion and establish its Paley-Wiener-Zygmund representation.

Let $\{B(t):t\in\R\}$ be a standard two-sided Brownian motion. For $(H,t)\in(0,1)\times\R_+$ let
\begin{equation}\label{def:X_pm}
\begin{split}
X^+_H(t) & := \int_{-\infty}^0 \left[(t-s)^{H-\tfrac{1}{2}} - (-s)^{H-\tfrac{1}{2}}\right]{\rm d}B(s) + \int_0^t(t-s)^{H-\tfrac{1}{2}}{\rm d}B(s), \\
X^-_H(t) & := - \int_0^t s^{H-\tfrac{1}{2}}{\rm d}B(s) - \int_t^\infty \left[s^{H-\tfrac{1}{2}}-(s-t)^{H-\tfrac{1}{2}}\right]{\rm d}B(s).
\end{split}
\end{equation}
Note that in case $H=\tfrac{1}{2}$ we have $X^+_{1/2}(t) = B(t)$, $X^-_{1/2}(t) = -B(t)$. Furthermore, for $\apm :=(\ap,\am)\in\R^2_0$, with $\R^2_0 := \R^2\setminus\{(0,0)\}$ put
\begin{equation}\label{def:X_apm}
X_H^\apm(t) := \ap X_H^+(t) + \am X_H^-(t)
\end{equation}
Finally, for any ${\bf c}\in\R_0^2$, we define the (standardized) Linear Fractional Brownian motion $\{B_H^\apm(t): (H,t)\in(0,1)\times\R_+\}$, where
\begin{equation}\label{def:B}
B^\apm_H(t) := \frac{X^\apm_H(t)}{\sqrt{V^\apm_H}}, \qquad V^\apm_H := \var X^\apm_H(1).
\end{equation}
Now, according to  \cite[Lemma~4.1]{StoevTaqqu06}  we have
\begin{equation}\label{def:V}
V^\apm_H := C^2_H \cdot \left[\left((\ap  + \am )\cos\big(\tfrac{\pi(H +1/2)}{2}\big)\right)^2 + \left((\ap  - \am )\sin\big(\tfrac{\pi(H +1/2)}{2}\big)\right)^2\right],
\end{equation}
where
\begin{equation}\label{def:C}
C^2_H := \frac{\Gamma(\tfrac{1}{2}+H)\Gamma(2-2H)}{2H\Gamma(\tfrac{3}{2}-H)}.
\end{equation}

We emphasize that $B_H(t)$ is Gaussian field with well-known covariance structure, i.e. for each $\apm\in\R^2_0$, the value of
\begin{equation}\label{eq:general_covariance_structure}
\cov(B^\apm_{H_1}(t_1),B^\apm_{H_2}(t_2)), \qquad (H_1,t_1),(H_2,t_2)\in(0,1)
\times\R^2_+
\end{equation}
is known, see \cite[Theorem~4.1]{StoevTaqqu06} (we do not write it here because the formula is quite involved and we don't use it). While for each $H\in(0,1)$, the process $\{B^\apm_H(t):t\in\R_+\}$ is an $H$-fBm (therefore its law is independent of the choice of the pair $\apm$), the covariance structure \eqref{eq:general_covariance_structure} of the entire field varies for different $\apm$; see \cite{StoevTaqqu06}. In orther words, different choices of $\apm$ will provide different couplings between the fBms. 
The case $\apm = (1,0)$ corresponds to the fractional Brownian field introduced by Mandelbrot \& van Ness  in \cite{mandelbrot1968fractional} (note that in this case we have $V^\apm_H = C^2_H$). We remark that the representation \eqref{def:B} was recently rediscovered in \cite{kordzakhia2018limit}.

Following \cite{bisewski2021derivatives}, we use the \textit{Paley-Wiener-Zygmund} (PWZ) representation of processes $\{\tilde X^+_H(t) : (H,t)\in(0,1)\times\R_+\}$ and $\{\tilde X^-_H(t) : (H,t)\in(0,1)\times\R_+\}$ defined in \eqref{def:X_pm}
\begin{equation}\label{eq:Xtilde_pm}
\begin{split}
\tilde X_H^+(t) & = t^{H-\tfrac{1}{2}}\cdot B(t) -(H-\tfrac{1}{2}) \cdot\int_0^t (t-s)^{H-\tfrac{3}{2}}\cdot\big(B(t)-B(s)\big)\,{\rm d}s\\
& \qquad\qquad + (H-\tfrac{1}{2}) \cdot\int_{-\infty}^0 \big[(t-s)^{H-\tfrac{3}{2}} - (-s)^{H-\tfrac{3}{2}}\big]\cdot B(s)\,{\rm d}s,\\
\tilde X_H^-(t) & = -t^{H-\tfrac{1}{2}}\cdot B(t) + (H-\tfrac{1}{2})\cdot \int_0^t s^{H-\tfrac{3}{2}}\cdot B(s)\,{\rm d}s\\
& \qquad\qquad + (H-\tfrac{1}{2})\cdot \int_t^\infty \left[s^{H-3/2} - (s-t)^{H-3/2}\right] \cdot\big(B(s)-B(t)\big)\,{\rm d}s.
\end{split}
\end{equation}

\begin{proposition}\label{prop:continuous_modification_pm}
$\{\tilde X^\pm_H(t): (H,t)\in(0,1)\times\R_+\}$ is a continuous modification of $\{X^\pm_H(t): (H,t)\in(0,1)\times\R_+\}$.
\end{proposition}

Now let us define the counterpart of the process $X^\apm_H(t)$ from Eq.~\eqref{def:X_apm}, that is, for every $\apm:=(\ap,\am)\in\R^2_0$ define stochastic process $\{\tilde X^\apm_H(t): (H,t)\in(0,1)\times\R_+\}$, where
\begin{equation}\label{def:PWZ}
\tilde X^\apm_H(t) = \ap \tilde X_H^+(t) + \am \tilde X_H^-(t)
\end{equation}

\begin{corollary}\label{coro:continuous_modification}
$\{\tilde X^\apm_H(t): (H,t)\in(0,1)\times\R_+\}$ is a continuous modification of $\{X^\apm_H(t): (H,t)\in(0,1)\times\R_+\}$.
\end{corollary}

For completeness, we give a short proof of Proposition \ref{prop:continuous_modification_pm} below.
\begin{proof}[Proof of Proposition \ref{prop:continuous_modification_pm}]
In \cite[Proposition~2]{bisewski2021derivatives} it was shown that $\{\tilde X^+_H(t): (H,t)\in(0,1)\times\R_+\}$ is a continuous modification of $\{X^+_H(t): (H,t)\in(0,1)\times\R_+\}$. Showing the sample path continuity of $\{\tilde X^-_H(t): (H,t)\in(0,1)\times\R_+\}$ is analogous to showing sample path continuity of $\tilde X^+$ which was done in \cite[Proposition~3]{bisewski2021derivatives}. Finally, due to \cite[Lemma~2]{bisewski2021derivatives}, for any $(H,t)\in(0,1)\times\R_+$ we have
\begin{align*}
\tilde X^-_H(t) =  X^-_H(t) \qquad \text{a.s.}
\end{align*}
This shows that $\tilde X^-$ is a modification of $X^-$ and concludes the proof.
\end{proof}

%---------------------------------------------------------
%---------------------------------------------------------
%---------------------------------------------------------

\subsection{Joint density of the supremum of drifted Brownian Motion and its time}\label{s:joint_density}

In this section we recall the formulae for the joint density of the supremum of (drifted) Brownian motion over $[0,T]$ and its time due to \cite{shepp1979joint}. This section relies heavily on \cite[Section~2]{bisewski2021derivatives}.

Let $\{B(t):t\in\R_+\}$ be a standard Brownian motion. For any $T>0$ and $a\in\R$ consider the supremum of drifted Brownian motion its time, i.e.
\begin{equation}\label{def:M_and_tau_H=brown}
M_{1/2}(T,a) := \sup_{t\in[0,T]} \{B(t) - at\}, \quad \tau_{1/2}(T,a) := \argmax_{t\in[0,T]}\{B(t) - at\}
\end{equation}
and its expectations
\begin{equation}\label{def:expectations_H=brown}
\mathscr M_{1/2}(T,a) := \e\left(M_{1/2}(T,a)\right), \quad \mathcal E_{1/2}(T,a) := \e\left(\tau_{1/2}(T,a)\right)
\end{equation}
In the following, let $p(t,y;T,a)$ be the joint density of $(\tau_{1/2}(T,a), M_{1/2}(T,a))$, i.e.
\begin{equation}
p(t,y; T,a) := \frac{\p(\tau_{1/2}(T,a)\in{\rm d}t, M_{1/2}(T,a)\in{\rm d}y)}{{\rm d}t\,{\rm d}y}
\end{equation}
When $T\in(0,\infty)$ and $a\in\R$ then
\begin{equation}\label{eq:density_T}
p(t,y;T,a) = \frac{y\exp\left\{-\tfrac{(y+ta)^2}{2t}\right\}}{\pi t^{3/2}\sqrt{T-t}} \left(e^{-a^2(T-t)/2} + a\sqrt{\tfrac{\pi(T-t)}{2}}~\erfc\Big(-a\sqrt{\tfrac{T-t}{2}}\Big)\right)
\end{equation}
for $t\in(0,T)$ and $y>0$. When $a>0$, then the pair $(\tau_{1/2}(\infty,a), M_{1/2}(\infty,a))$ is well-defined with
\begin{equation}\label{eq:density_infty}
p(t,y;\infty,a) = \frac{\sqrt{2}\,ay\exp\left\{-\tfrac{(y+ta)^2}{2t}\right\}}{t^{3/2}\sqrt{\pi}},
\end{equation}
for $t>0$ and $y>0$.
\iffalse For convenience we also present the density of the pair in the driftless case, i.e. $a=0$,
\begin{equation}\label{eq:density_zero}
p(t,y,T,0) = \frac{y\exp\{-y^2/(2t)\}}{\pi t^{3/2}\sqrt{T-t}}
\end{equation}
\fi
\begin{proposition}\label{prop:tau_expectation}
It holds that
\begin{itemize}
\item if $T\in(0,\infty)$ and $a\neq 0$ then
\begin{align*}
\mathscr M_{1/2}(T,a) & = \frac{1}{2a}\left(-a^2T + (1+ a^2T)\erf\Big(a\sqrt{\tfrac{T}{2}}\Big)+ \sqrt{\frac{2T}{\pi}}\cdot ae^{-a^2T/2}\right) \\
\mathcal E_{1/2}(T,a) & = 
\frac{1}{2a^2}\left(a^2T + (1-a^2T)\erf\Big(a\sqrt{\tfrac{T}{2}}\Big) - \sqrt{\frac{2T}{\pi}}\cdot ae^{-a^2T/2}\right);
\end{align*}
\item if $a>0$ then $\mathscr M_{1/2}(\infty,a) = \frac{1}{2a}$ and $\mathcal E_{1/2}(\infty,a) = \frac{1}{2a^2}$;
\item if $T\in(0,\infty)$ then $\mathscr M_{1/2}(T,0) = \sqrt{\frac{2T}{\pi}}$ and $\mathcal E_{1/2}(T,0) = \frac{T}{2}$.
\end{itemize}
\end{proposition}
\begin{proof}
The fomula for $\mathscr M_{1/2}(T,a)$ can be obtained from the Laplace transform of $M_{1/2}(T,a)$, see formulas 1.1.1.3 and  2.1.1.3 in \cite{borodin2002handbook}. The formula for $\mathcal E_{1/2}(T,a)$ could similarly be obtained from the Laplace transform of $\tau_{1/2}(T,a)$. However, the numerical calculations indicate that the formulas for Laplace transforms 1.1.12.3, 2.1.12.3 in \cite{borodin2002handbook} are incorrect. Therefore, we provide our own derivation of $\mathcal E_{1/2}(T,a)$ in Appendix \ref{appendix:calculations}.
\end{proof}

Finally, we introduce a certain functional of Brownian motion, which plays an important role in this manuscript. It is noted that its special case ($H=\tfrac{1}{2}$) appeared in \cite[Eq.~(11)]{bisewski2021derivatives}. In what follows let $Y(t) := B(t)-at$
\begin{align}\label{def:I}
I_H(t,y) & := \e\left(\int_0^t(t-s)^{H-3/2}(Y(t)-Y(s)){\rm d}s \,\Big\vert\, (\tau_{1/2}(T,a),M_{1/2}(T,a)) = (t,y)\right)
%\\ & = \e\Big(\int_0^t s^{H-\tfrac{3}{2}}W_{t,y}(s){\rm d}s\Big)
\end{align}
Following \cite{bisewski2021derivatives}, we recognize that the conditional distribution of the process $\{Y(t)-Y(t-s):s\in[0,t]\}$ given $(\tau_{1/2}(T,a),M_{1/2}(T,a)) = (t,y)$ follows the law of the generalized 3-dimensional Bessel bridge from $(0,0)$ to $(t,y)$. Therefore, $I_H(t,y)$ can be thought of an expected value of a certain `Brownian area', see \cite{janson2007brownian} for a survey on Brownian areas.
It turns out that the function $I_H(t,y)$ can be explicitly calculated. In the following, $U(a,b,z)$ is the Tricomi's confluent hypergeometric function, see \eqref{def:tricomi} in Appendix~\ref{appendix:special_functions}.
\begin{lemma}\label{lem:IH}
If $H\in(0,\tfrac{1}{2})\cup(\tfrac{1}{2},1)$, and $t,y>0$, then
\begin{align*}
I_H(t,y) = \frac{t^{H+1/2}}{y(H-\tfrac{1}{2})(H+\tfrac{1}{2})} \left(1-\frac{\Gamma(H)}{\sqrt{\pi}}U\left(H-\tfrac{1}{2},\tfrac{1}{2},\tfrac{y^2}{2t}\right)\right) + \frac{t^{H-1/2}y}{H+\tfrac{1}{2}}.
\end{align*}
\end{lemma}
The derivation of the result in Lemma \ref{lem:IH} is purely calculational. For completeness, we provide it in Appendix \ref{appendix:calculations}.

\begin{proof}[Proof of Lemma \ref{lem:IH}]
Let
\begin{equation*}
g(x,s;t,y) := \frac{\p\left(Y(t)-Y(t-s)\in {\rm d}x \mid (\tau_{1/2}(T,a),M_{1/2}(T,a)) = (t,y)\right)}{{\rm d}x}.
\end{equation*}
From \cite[Eq.~(10)]{bisewski2021derivatives} we have
\begin{equation*}
g(x,s;t,y)  = \frac{\frac{x}{s^{3/2}} \exp\{-\frac{x^2}{2s}\}}{\frac{y}{t^{3/2}}\exp\{-\frac{y^2}{2t}\}} \cdot \frac{1}{\sqrt{2\pi(t-s)}}\left[e^{-\frac{(y-x)^2}{2(t-s)}} - e^{-\frac{(y+x)^2}{2(t-s)}} \right]
\end{equation*}
for $x>0$. Using Fubini-Tonelli theorem we have
\begin{equation}\label{eq:IH_integral}
I_H(t,y) = \int_0^t\int_0^\infty s^{H-3/2}x\cdot g(x,s;t,y)\,{\rm d}x{\rm d}s.
\end{equation}
The rest of the proof is purely calculational. For completeness, it is given in Appendix \ref{appendix:calculations}.
\end{proof}

%----------------------------------------------------
%----------------------------------------------------
%----------------------------------------------------

\section{Main results}\label{s:main_results}

Let $\{B(t) : t\in\R\}$ be a standard, two-sided Brownian motion. Consider the PWZ representation of LFBM with parameter $\apm=\apm\in\R^2_0$, that is $\{\tilde B^\apm_H(t):t\in\R^+\}$, where
\begin{align*}
\tilde B^\apm_H(t) := \frac{\tilde X^\apm_H(t)}{\sqrt{V^\apm_H}}
\end{align*}
with $\tilde X^\apm_H(t)$ defined in \eqref{def:PWZ} and $V^\apm_H$ defined in \eqref{def:V}. Then, according to Corollary~\ref{coro:continuous_modification}, $\{\tilde B^\apm_H(t):H\in(0,1)\times\R_+\}$ is a continuous modification of $\{B^\apm_H(t):H\in(0,1)\times\R_+\}$ and therefore, for each fixed $H$, $\{\tilde B^\apm_H(t): t\in\R_+\}$ is a fractional Brownian motion with Hurst index $H$. It is noted that all processes $\tilde B^\apm_H$ live on the same probability space and are, in fact, defined pathwise for every realization of the driving Brownian motion.

\iffalse
Notice that each trajectory $t\mapsto\tilde X^+_H(t)$ and $t\mapsto\tilde X^-_H(t)$ is completely defined by the trajectory of Brownian motion $B$. In the previous section we have seen that for all $\apm\in\R^2_0$, the process 
has the law of fractional Brownian motion with Hurst index $H\in(0,1)$. Again, we emphasize that each trajectory $t\mapsto\tilde B^\apm(t)$ is completely defined by the underlying trajectory of Brownian motion $t\mapsto B(t)$. Changing the values of the pair $\apm$ does not affect the overall law of $\tilde B^\apm_H(t)$ but it does change it pathwise.
\fi
For each $\apm\in\R^2_0$, $a\in\R$ and $T>0$ define
\begin{equation}\label{def:M_and_tau}
M^\apm_H(T,a) := \sup_{t\in[0,T]}\left\{B_H^\apm(t) - at\right\}, \quad \tau^\apm_H(T,a) := \argmax_{t\in[0,T]} \left\{B_H^\apm(t) - at\right\},
\end{equation}
which is the supremum of drifted Brownian motion with paremeter $\apm$ and its location. Now we define the expected expected values of these random variables
\begin{equation}\label{def:expectations}
\mathscr M_H(T,a) := \e\left( M_H^\apm(T,a)\right), \quad \mathcal E_H(T,a) := \e \left(\tau_H^\apm(T,a)\right).
\end{equation}
Recall that the random variables and their expectations above were already defined for $H=\tfrac{1}{2}$ in Eq.~\eqref{def:M_and_tau_H=brown} and Eq.~\eqref{def:expectations_H=brown}. Notice how $\mathscr M_H(T,a)$ and $\mathcal E_H(T,a)$ do not dependent on $\apm$, because as $\apm$ vary, the joint law of the supremum of its location does not change. Now, for each $\apm\in\R^2_0$ define
\begin{align}\label{def:mHapm}
m^\apm_H(T,a) &:= \e\left(\tilde B^\apm_H(\tau_{1/2}(T,a)) - a\tau_{1/2}(T,a)\right),
\end{align}
which, in words, is the expected value of drifted fBm with parameter $\apm$ evaluated at time of the supremum of the driving Brownian motion. Clearly, this yields
$$\mathscr M_H(T,a) \geq m^\apm_H(T,a)$$
and $m^\apm_{1/2}(T,a) = \mathscr M_{1/2}(T,a)$ for every $\apm\in\R^2_0$. We further maximize our lower bound by taking the supremum over all feasible pairs $\apm$ and define
\begin{equation}\label{def:mH}
m_H(T,a) := \sup_{\apm\in\R^2_0} m^\apm_H(T,a).
\end{equation}
It turns out that the value of $m_H(T,a)$ can be found explicitly. Before showing the formula in Proposition~\ref{prop:mH} we define a useful functional
\begin{equation}\label{def:J}
\mathcal J_H(T,a) := \e\left(X^+_H(\tau_{1/2}(T,a))\right).
\end{equation}
In the following $\gamma(\alpha,z)$ is the incomplete Gamma function, see \eqref{def:incomplete_gamma} in Appendix \ref{appendix:special_functions}.
\begin{proposition}\label{prop:Jvalue}
It holds that
\begin{equation}\label{eq:Jvalue}
\mathcal J_H(T,a) =
\begin{cases}
\frac{2^{H}}{\sqrt{2\pi}(H+\tfrac{1}{2})} \cdot |a|^{-2H}\gamma(H,\tfrac{a^2T}{2}), & a\neq 0,\ T\in(0,\infty)\\
\frac{2^{H}\Gamma(H)}{\sqrt{2\pi}(H+\tfrac{1}{2})} \cdot |a|^{-2H}, & a > 0, \ T=\infty \\
\frac{T^{H}}{\sqrt{2\pi}H(H+\tfrac{1}{2})}, & a=0, \ T\in(0,\infty)
 \end{cases}
\end{equation}
\end{proposition}
Finally, we can show that
\begin{proposition}\label{prop:mH}
If $a\in\R$ and $T>0$ or $a>0$ and $T=\infty$, then
\begin{equation}\label{eq:m_formula}
m^\apm_H(T,a) = \frac{\ap-\am }{\sqrt{V^\apm_H}}\cdot \mathcal J_H(T,a) - a\mathcal E_{1/2}(T,a)
\end{equation}
and
\begin{equation}
m_H(T,a) = m^{(1,-1)}_H(T,a) = \frac{\mathcal J_H(T,a)}{C_H\sin\big(\tfrac{\pi(H +1/2)}{2}\big)} - a\, \mathcal E_{1/2}(T,a),
\end{equation}
with $V^\apm_H$ and $C_H$ defined in \eqref{def:V} and \eqref{def:C}, respectively.
\end{proposition}
We emphasize the the values of functions $\mathcal J_H(T,a)$ and $\mathcal E_{1/2}(T,a)$ are known explicitly, see Proposition \ref{prop:Jvalue} and Proposition \ref{prop:tau_expectation} respectively. Additionally, we write down the formulas for $m_H(T,a)$ in two special cases below.
\begin{itemize}
\item if $T\in(0,\infty)$, then $\displaystyle m_H(T,0) = \frac{T^H}{\sqrt{2\pi}\,C_HH(H+\tfrac{1}{2})\sin\big(\tfrac{\pi}{2}(H+\tfrac{1}{2})\big)}$;
\item if $a>0$, then $\displaystyle m_H(\infty,a) = \frac{2^{H}\Gamma(H)}{\sqrt{2\pi}C_Ha^{2H}(H+\tfrac{1}{2})\sin\big(\tfrac{\pi}{2}(H+\tfrac{1}{2})\big)} - \frac{1}{2a}$.
\end{itemize}

Interestingly, we can further improve the lower bound derived in Proposition~\ref{prop:mH} simply by using self-similarity of fractional Brownian motion, i.e. for any $\rho>0$ we have
\begin{align*}
\mathscr M_H(T,a) = \rho^{-H}\mathscr M_H(\rho T,\rho^{H-1}a),
\end{align*}
which also holds for $a>0$ and $T=\infty$, i.e. $\mathscr M_H(\infty,a) = \rho^{-H}\mathscr M_H(\infty,\rho^{H-1}a)$. Therefore, we present our final result in the corollary below. In the following, let
\begin{equation}\label{def:underline_MH}
\underline{\mathscr M}_H(T,a) := \sup_{\rho>0}\left\{\rho^{-H} m_H(\rho T,\rho^{H-1}a)\right\}
\end{equation}
\begin{corollary}\label{coro:MH}
For any $H\in(0,1)$, $T>0$, and $a\in\R$ it holds that
\begin{align*}
\mathscr M_H(T,a) \geq \underline{\mathscr M}_H(T,a).
\end{align*}
In particular,
\begin{itemize}
\item[(i)] if $T\in(0,\infty)$, then $\displaystyle \underline{\mathscr M}_H(T,0) = \frac{T^H}{\sqrt{2\pi}\,C_HH(H+\tfrac{1}{2})\sin\big(\tfrac{\pi}{2}(H+\tfrac{1}{2})\big)}$;
\item[(ii)] if $a>0$, then $\displaystyle \underline{\mathscr M}_H(\infty,a) = \frac{1-H}{2aH} \cdot \left(\frac{2^{H+1}a^{1-2H}H\Gamma(H)}{\sqrt{2\pi}C_H(H+\tfrac{1}{2})\sin\big(\tfrac{\pi}{2}(H+\tfrac{1}{2})\big)}\right)^{1/(1-H)}$.
%\item[(ii)] if $a>0$, then $\displaystyle \underline{\mathscr M}_H(\infty,a) = \frac{1-H}{2aH}\left(\frac{2aH\mathcal J_H(\infty,a)}{C_H\sin\big(\tfrac{\pi}{2}(H+\tfrac{1}{2})\big)}\right)^{1/(1-H)} = \frac{1-H}{2aH} \cdot \left(\frac{2^{H+1}a^{1-2H}H\Gamma(H)}{\sqrt{2\pi}C_H(H+\tfrac{1}{2})\sin\big(\tfrac{\pi}{2}(H+\tfrac{1}{2})\big)}\right)^{1/(1-H)}$.
\end{itemize}
\end{corollary}
It is noted that $\underline{\mathscr M}_H(T,0) = m_H(T,0)$. In general, the solution to the optimization problem \eqref{def:underline_MH} can be found numerically because the explicit formula for $m_H(T,a)$ is known, cf. Proposition~\ref{prop:mH}.

\subsection{Secondary results}
Before ending this section, we would to present two immediate corollaries, which are implied by our main results. First result describes the asymptotic behavior of $\mathscr M_H(T,0)$, as $H\downarrow0$ while the second result pertains the evaluation of the derivative of the expected supremum $\mathscr M_H(T,a)$ at $H=\tfrac{1}{2}$.

\subsubsection*{Behavior of $\underline{\mathscr M}_H(1,0)$, as $H\downarrow 0$.} 
Using the formula for $\underline{\mathscr M}_H(T,0)$ from Corollary~\ref{coro:MH}(i), it is easy to obtain the following result.
\begin{corollary}\label{coro:smallH}
It holds that
\begin{align*}
\underline{\mathscr M}_H(1,0) \sim \frac{2}{\sqrt{\pi H}}, \quad H\to0.
\end{align*}
\end{corollary}
The asymptotic lower bound in Corollary~\ref{coro:smallH} is over \emph{5 times larger} than the corresponding bound derived in \cite[Theorem~2.1(i)]{borovkov2017bounds}, where it was shown that $\mathscr M_H(1,0) \geq (4H\pi e\log(2))^{-1/2}$ for all $H\in(0,1)$. Moreover, together with \cite[Corollary~2]{borovkov2018new}, our result implies that
$$ 1.128 \leq H^{-1/2}\mathscr M_H(1,0) \leq 1.695$$
for all $H$ small enough. Determining whether the limit $H^{-1/2}\mathscr M_H(1,0)$, as $H\to0$ exists and finding its value remains an interesting open question.

\subsubsection*{Derivative of the expected supremum}
Using the formula for $\mathcal J_H(T,a)$ in Proposition \ref{prop:Jvalue} and recent findings in \cite{bisewski2021derivatives}, we are able to explicitly evaluate the derivative of the expected supremum $\frac{\partial}{\partial H}\mathscr M_H(T,a)$ at $H=\tfrac{1}{2}$. The proof of \cite[Theorem~1]{bisewski2021derivatives} implies that
\begin{align*}
\mathscr M'_{1/2}(T,a) := \frac{\partial}{\partial H}\mathscr M_H(T,a) \Big\vert_{H=1/2} = \frac{\partial}{\partial H}m^{(1,0)}_H(T,a) \Big\vert_{H=1/2}.
\end{align*}
Therefore, using the formula for $m^{(1,0)}_H(T,a)$ from Proposition \ref{prop:mH}, we find that $\mathscr M'_{1/2}(T,a) = \mathcal J_{1/2}(T,a)+\frac{\partial}{\partial H}\mathcal J_H(T,a)\vert_{H=1/2}$. In the following let $\gamma'(s,x) = \frac{\partial}{\partial s}\gamma(s,x)$.
\begin{corollary}\label{coro:derivative}
It holds that
\begin{align*}
\mathscr M'_{1/2}(T,a) = \frac{1}{\sqrt{\pi}|a|} \left(\log(2a^{-2})\gamma(\tfrac{1}{2},\tfrac{a^2T}{2})+\gamma'(\tfrac{1}{2},\tfrac{a^2T}{2})\right).
\end{align*}
\end{corollary}
It is noted that the continuous extension of the function $\mathscr M_{1/2}'(T,a)$ to $(T,0)$ and to $(\infty,a)$ agrees with Corollaries 1(i) and 1(ii) from \cite{bisewski2021derivatives}, respectively.

\iffalse

\subsubsection*{Anti-balanced case.} It was somewhat expected that the supremum over all feasible pairs $\apm$ is attained at $(1,-1)$ because this pair is maximizing the correlation between the driving Brownian motion at any time $t\in\R_+$, i.e.
\begin{align*}
\cov(B_H^\apm(t),B(t)) & = \frac{\ap\cov(X^+_H(t),B(t))+\am\cov(X^-_H(t),B(t))}{\sqrt{V^\apm_H}} \\
& = \frac{\ap\int_0^t(t-s)^{H-1/2}{\rm d}s-\am\int_0^t s^{H-1/2}{\rm d}s}{\sqrt{V^\apm_H}}\\
& = \frac{\ap-\am}{\sqrt{V^\apm_H}} \cdot \frac{t^{H+1/2}}{H+\tfrac{1}{2}},
\end{align*}
which is easily seen to be maximized at $\apm = (1,-1)$.

\subsubsection*{Approximation of the distribution of $\sup_{t\in[0,1]}B_H(t)$.} In the following, let
\[M_H := \sup_{t\in[0,1]}B_H(t), \quad \text{and} \quad M^\tau_H := B^{(1,-1)}_H(\tau(1,0))\]
From the plots Figure~\ref{fig:MH10)} it appears that $\e M_H \approx \e M_H^\tau$, when $H\approx 1/2$ but actually, this seems to hold true  for all $H<1/2$. We won't know for sure how well it behaves around 

\fi
%---------------------------------------------------------
%---------------------------------------------------------
%---------------------------------------------------------

\section{Numerical experiments}\label{s:numerical_experiments}

In this section we compare our theoretical lower bound $\underline{\mathscr M}_H(T,a)$ with Monte Carlo simulations.

In our experiments we use the circulant embedding method (also called Davies-Harte method \cite{davies1987tests}) for the simulation of fBm; see also \cite{dieker2004simulation} for various methods of simulation. Experiments were perfomed in Python and the code of the Davies-Harte method was adapted from \cite[Section~12.4.2]{kroese2015spatial}.
The method relies on the simulation of fBm on an equidistant grid of $n$ points, that is $(B(0), B(\tfrac{T}{n}), B(\tfrac{2T}{n}), \ldots, B(T))$. The resulting estimator has the expected value
\begin{align*}
\mathscr M^n_H(T,a) := \e\left(\sup_{t\in\mathcal T_n} \{B_H(t) - at\}\right),
\end{align*}
where $\mathcal T_n := \{0,\tfrac{T}{n},\tfrac{2T}{n},\ldots,T\}$.  Clearly, $\mathscr M^n_H(T,a) \leq \mathscr M_H(T,a)$, i.e. on average, the Monte Carlo estimator underestimates the ground truth, as the supremum is taken over a finite subset of $[0,T]$. Nonetheless, as $n\to\infty$ we should observe that $\mathscr M^n_H(T,a) \to \mathscr M_H(T,a)$.

In our experiments, we consider three different cases (i) $T=1$, $a=0$, (ii) $T=1$, $a=1$, and finally (iii) $T=\infty$, $a=1$. In each case the theoretical lower bound $\underline{\mathscr M}_H(T,a)$ is compared with the corresponding simulation results $\hat{\mathscr M}^n_H(T,a)$ for $n\in\{2^{10},2^{12},2^{14},2^{16}\}$ based on $20,000$ independent runs of Davies-Harte algorithm; in case (ii) we took $T=10$ because it is not possible to simulate the process over the infinite time horizon.

The results corresponding to cases (i)--(iii) are displayed in Figures \ref{fig:0}--\ref{fig:2}. Each curve is surrounded by its $95\%$ confidence interval depicted as a shaded blue area. The results are presented in two different scales.  We interpret the results on the figures on the left and on the right separately, in the following two paragraphs.

In the figures on the left, we compare the bound with the simulation results on the `high-level' for all $H\in(0.01,1)$. The blue dots at $H=\tfrac{1}{2}$ and $H=1$ correspond to the known values of $\mathscr M_{1/2}(T,a)$ and $\mathscr M_{1}(T,a)$ respectively; the value at $H=1$ in Figure~\ref{fig:2} is not shown because $\mathscr M_1(\infty,1) = \infty$. As expected, the value of $\hat{\mathscr M}_H^n(T,a)$ is increasing, as $n$ is increasing. The simulation results seem to roughly agree with the ground truth at $H=\tfrac{1}{2}$ and $H=1$, while the lower bound, agrees with the ground truth at $H=\tfrac{1}{2}$ by the definition, i.e. $\underline{\mathscr M}_{1/2}(T,a) = \mathscr M_{1/2}(T,a)$. On the `high-level', we can conclude that the lower bound $\underline{\mathscr M}_H(T,a)$ 
\begin{itemize}
\item[-] is close to the simulation results for all $H\approx \tfrac{1}{2}$,
\item[-] perfoms much better than the simulation results as $H\to0$, and
\item[-] performs worse when $H\to1$ (in fact, the bound seems to converge to $0$ there).
\end{itemize}

In the figures on the right, we compare the bound with the simulation results in the region $H\in(0.2,0.8)$. We plotted the relative error between the theoretical lower bound and the simulation results based on $n=2^{16}$ gridpoints, that is $$(\underline{\mathscr M}_H(T,a)-\hat{\mathscr M}_H^n(T,a))/\hat{\mathscr M}_H^n(T,a).$$ It is noted that if relative error is positive, then the theoretical lower bound yields a \emph{better approximation} to the ground truth than the Monte Carlo method, which is indicated by the green area below the curve on the plot. In this sense, we see that the theoretical lower bound outperforms the Monte Carlo simulations for $H\in(0,\tfrac{1}{2})$ in all three cases (i)--(iii). We remark that the value of the relative error at $H=\tfrac{1}{2}$ equals approximately $0.5\%$ in Figure~\ref{fig:0} and $0.7\%$ in Figure~\ref{fig:1}, which roughly agrees with \cite[Corollary~4.3]{bisewski2020zooming}, which states that
$$\mathscr M_{1/2}^n(T,a)-\mathscr M_{1/2}(T,a) \approx  - \sqrt{T} \cdot \frac{\zeta(1/2)}{\sqrt{2\pi n}} \approx -0.5826 \cdot \sqrt{T/n},$$
see also \cite[Theorem~2]{asmussen1995discretization} in case $a=0$.

\begin{figure}[h]
 \begin{adjustwidth}{0cm}{}
        \centering
       \includegraphics[width=.49\textwidth]{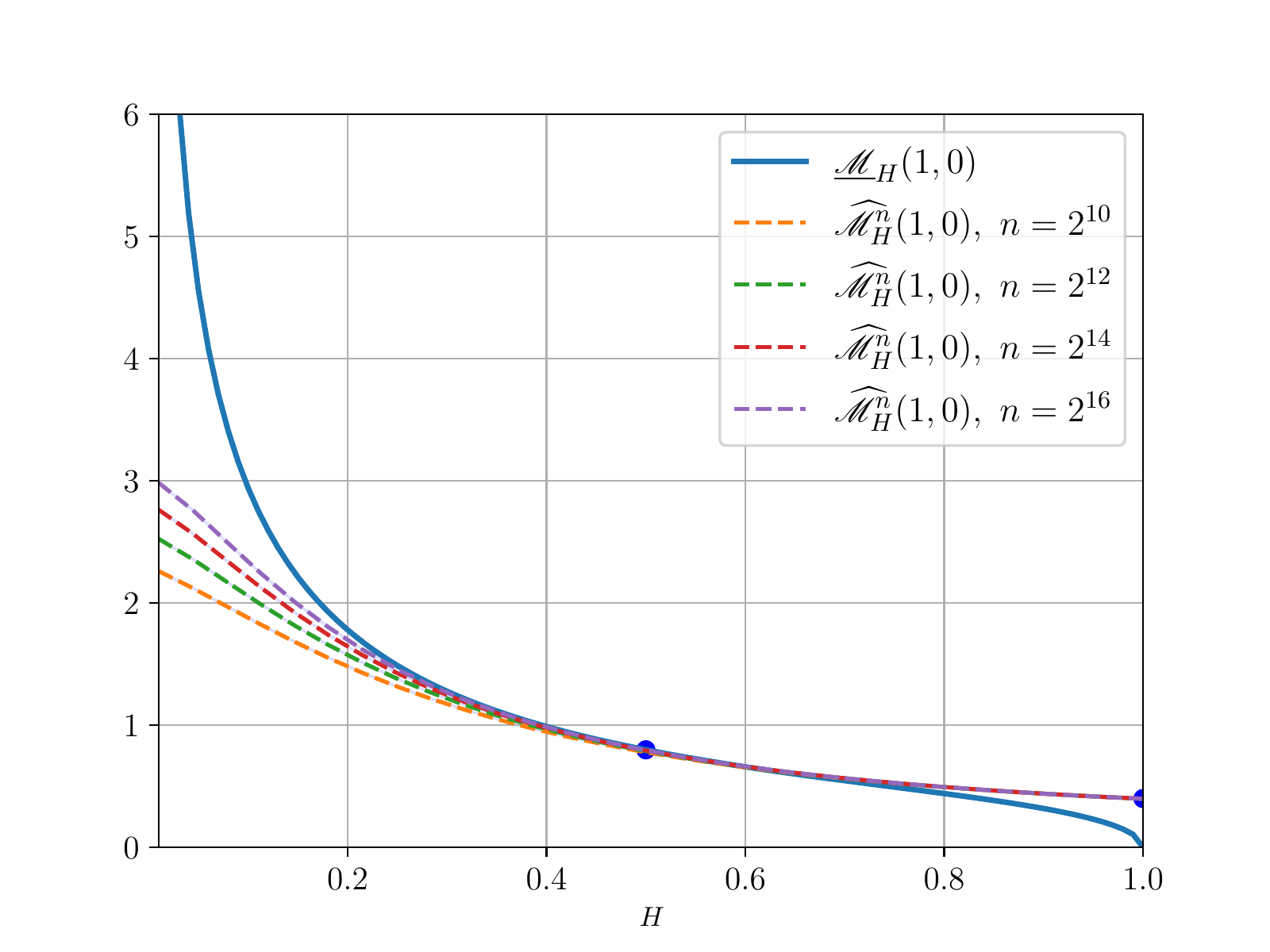}
       \includegraphics[width=.49\textwidth]{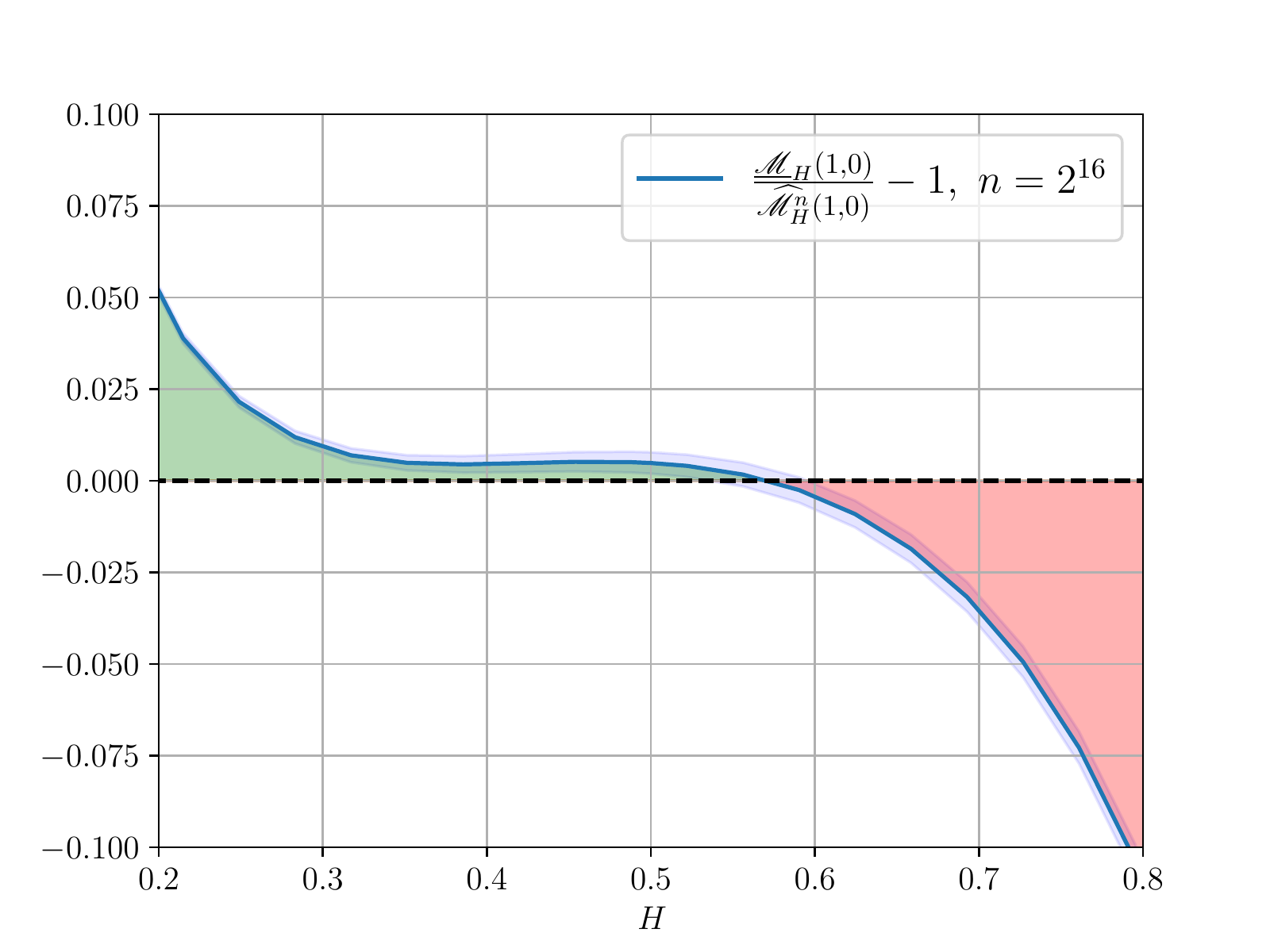}
      \caption{Numerical results in case $T=1, a=0$.
      }\label{fig:0}
       \end{adjustwidth}
\end{figure}

\begin{figure}[h]
 \begin{adjustwidth}{0cm}{}
        \centering
       \includegraphics[width=.49\textwidth]{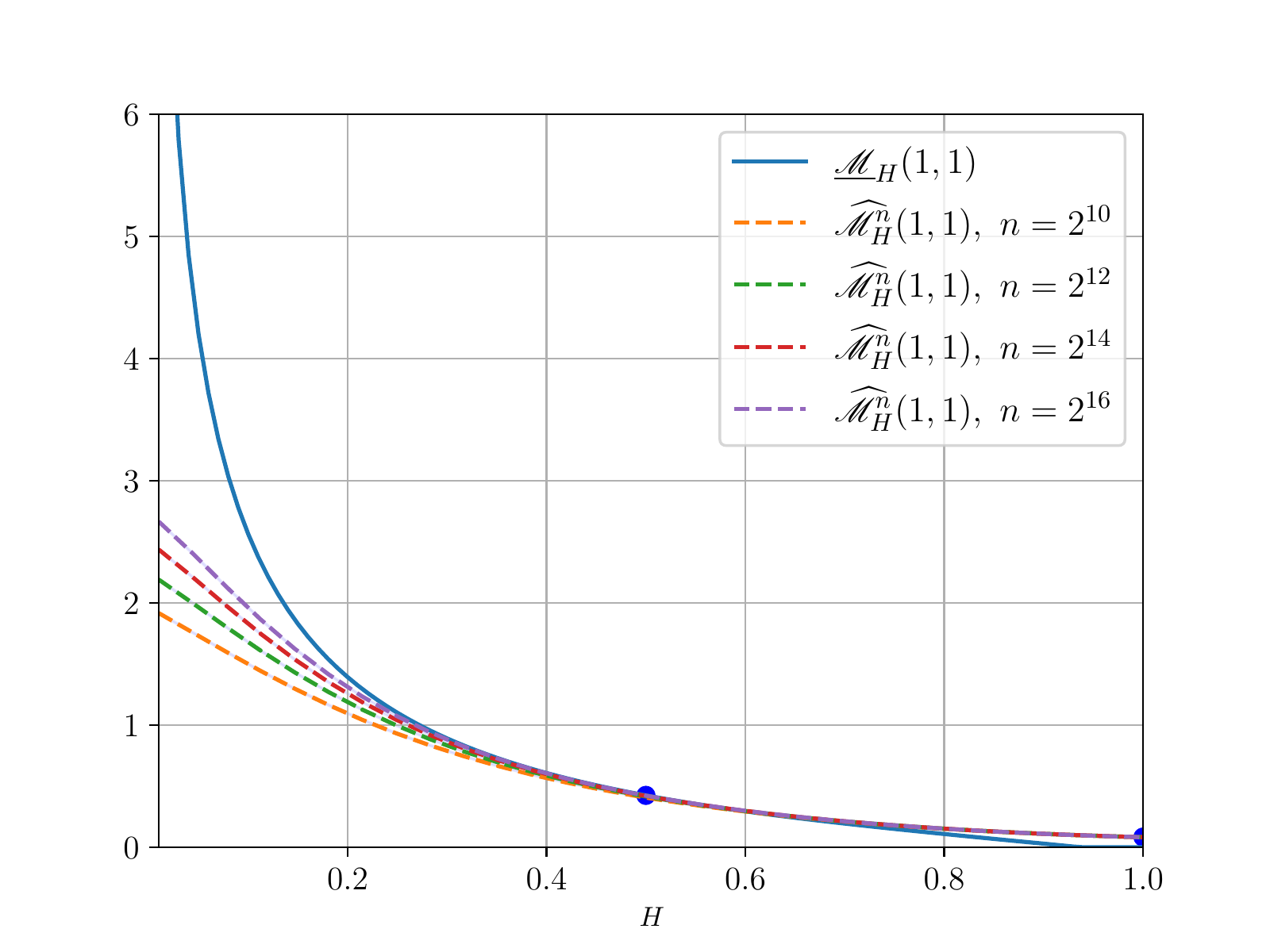}
       \includegraphics[width=.49\textwidth]{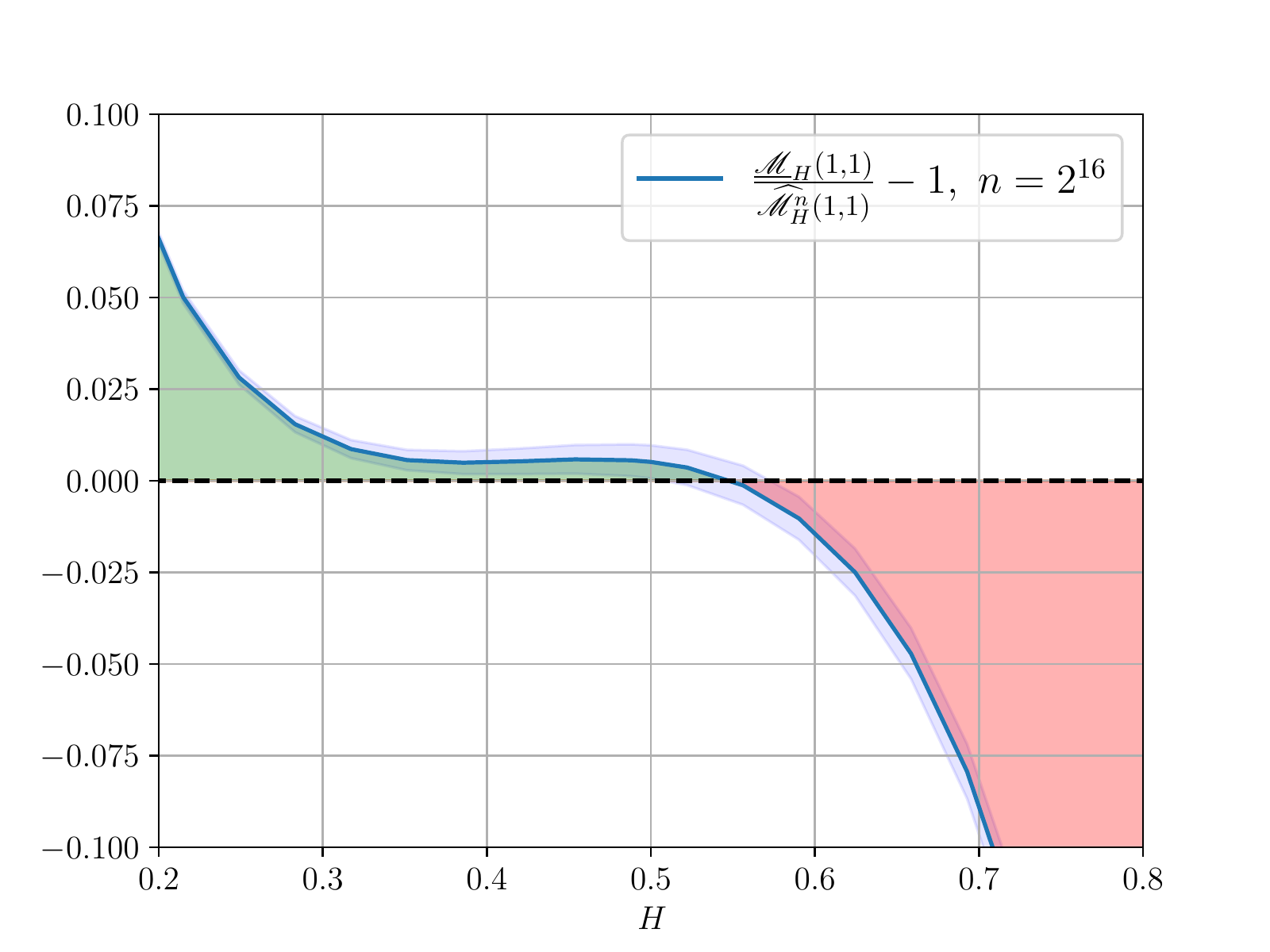}
      \caption{Numerical results in case $T=1, a=1$.
      }\label{fig:1}
       \end{adjustwidth}
\end{figure}

\begin{figure}[h]
 \begin{adjustwidth}{0cm}{}
        \centering
       \includegraphics[width=.49\textwidth]{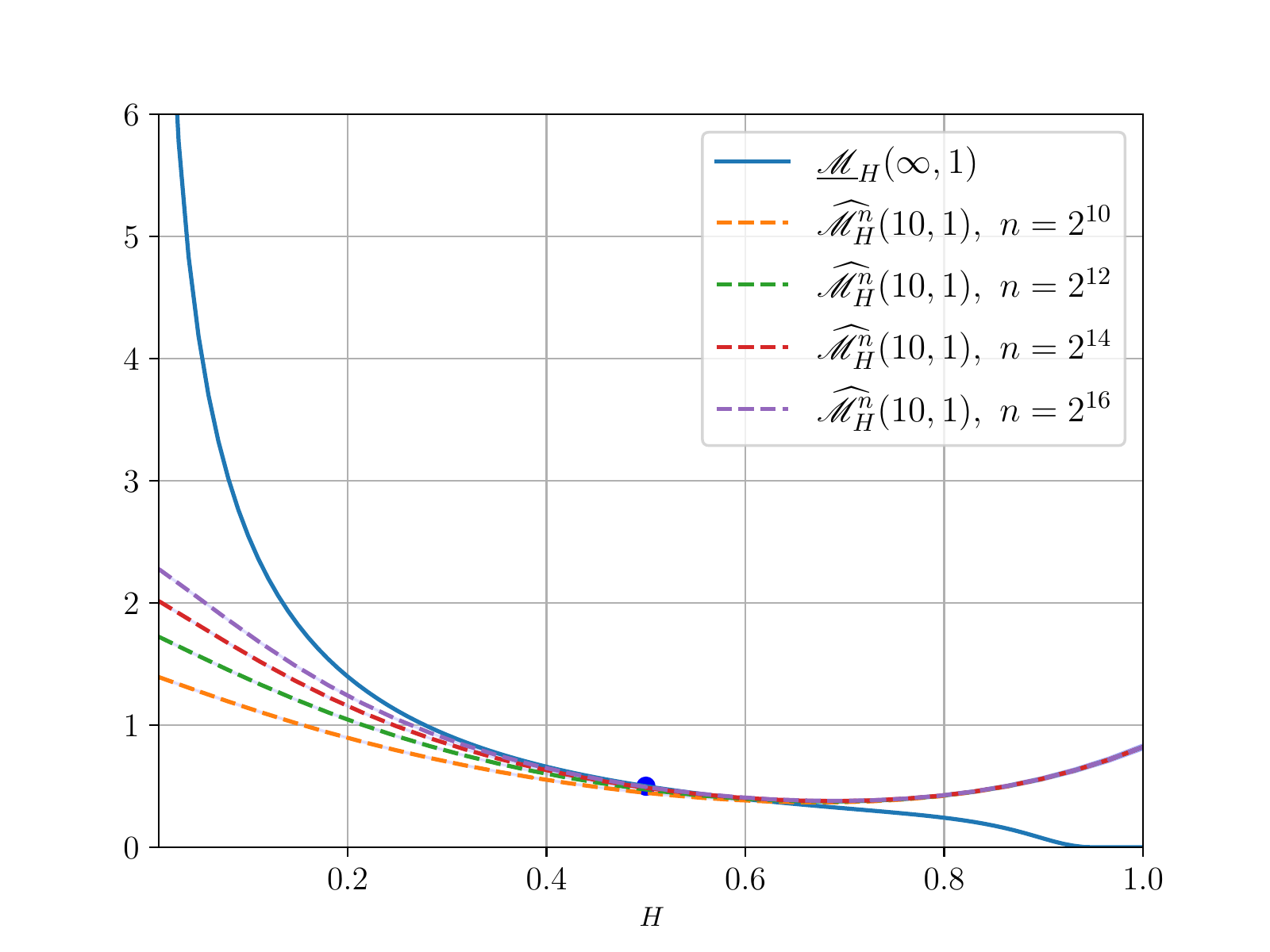}
       \includegraphics[width=.49\textwidth]{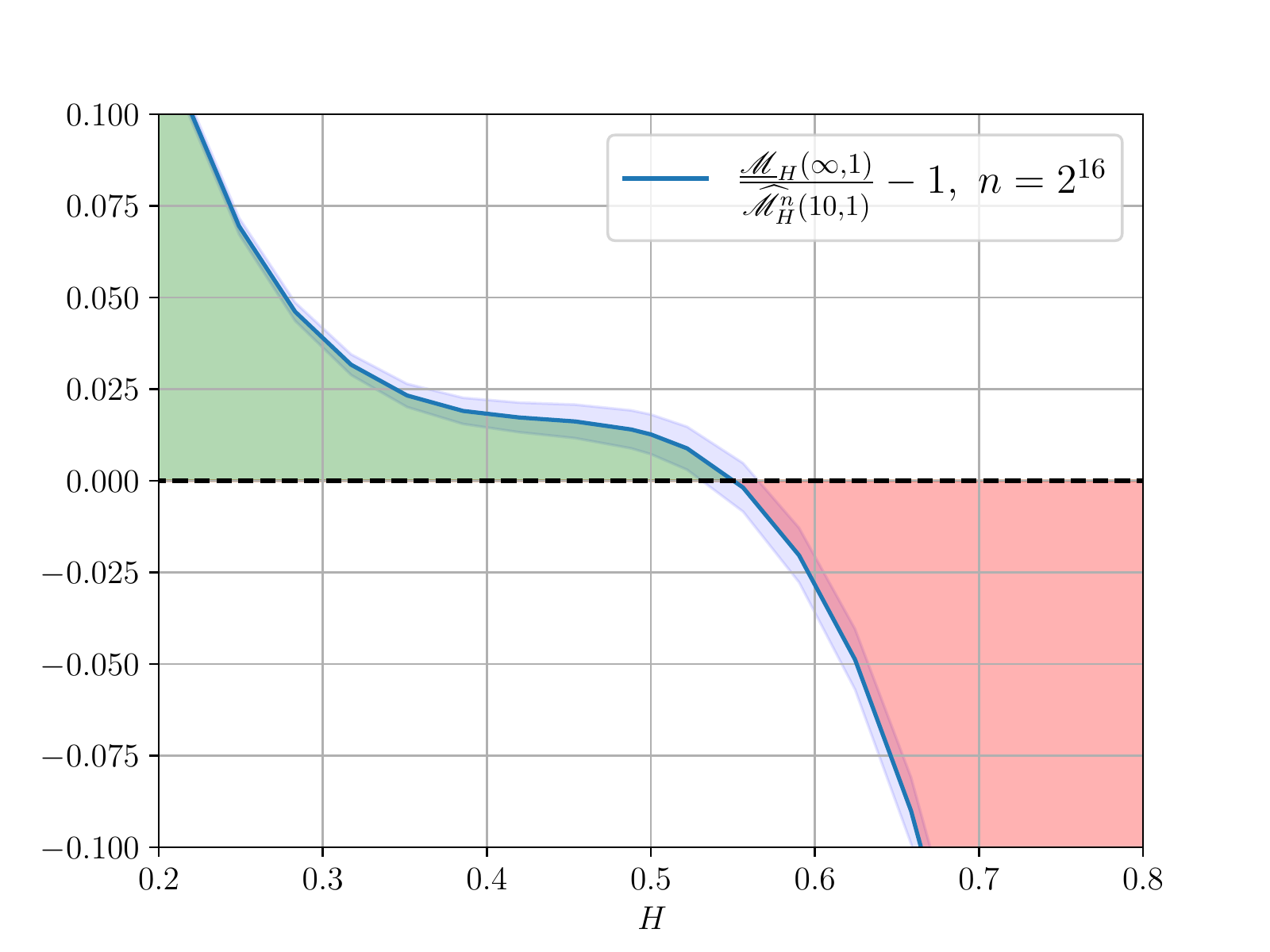}
      \caption{Numerical results in case $T=\infty, a=1$.
      }\label{fig:2}
       \end{adjustwidth}
\end{figure}

%---------------------------------------------------------
%---------------------------------------------------------
%---------------------------------------------------------

\section{Proofs}\label{s:proofs}
In the following, let
\begin{equation}\label{def:J_apm}
\mathcal J^+_H(T,a) := \e\left\{X^+_H(\tau_{1/2}(T,a))\right\}, \quad \mathcal J^-_H(T,a) := \e\left\{X^-_H(\tau_{1/2}(T,a))\right\}.
\end{equation}

\begin{lemma}\label{lem:J_antisymmetric}
If $T>0$ and $a\in\R$ then $\mathcal J^-_H(T,a) =  -\mathcal J^+_H(T,-a).$
\end{lemma}

\begin{proof}[Proof of Lemma \ref{lem:J_antisymmetric}]
For breviety we write $\tau:=\tau(T,a)$ and put $Y(t):=B(t)-at$. See that according to the PWZ representation of $\tilde X^+_H$ in Eq.~\eqref{eq:Xtilde_pm}, we have
\begin{align}
\nonumber \e\left(\tilde X^+_H(\tau)\right) = \e\left( \tau^{H-\tfrac{1}{2}}B(\tau) -(H-\tfrac{1}{2}) \cdot\int_0^\tau (\tau-s)^{H-\tfrac{3}{2}}\cdot\big(B(\tau)-B(s)\big)\,{\rm d}s\right) \\
\label{eq:Xplus_exp} = \e\left( \tau^{H-\tfrac{1}{2}}Y(\tau) + \frac{a\tau^{H+1/2}}{H+1/2} -(H-\tfrac{1}{2}) \cdot\int_0^\tau (\tau-s)^{H-\tfrac{3}{2}}\cdot\big(Y(\tau)-Y(s)\big)\,{\rm d}s\right),
\end{align}
where in the second line we simply substituted $B(t) = Y(t)+at$. Furthermore, notice that for any $T\in(0,\infty)$, the PWZ representation of $\tilde X_H^-(\tau)$ from Eq.~\eqref{eq:Xtilde_pm} can be rewritten as 
\begin{align*}
\tilde X_H^-(\tau) & := -(T-\tau)^{H-1/2}\big(B(\tau)-B(T)\big) + (H-\tfrac{1}{2})\cdot\int_\tau^T(s-\tau)^{H-3/2}(B(\tau)-B(s)){\rm d}s\\
& + (H-\tfrac{1}{2})\cdot\int_0^T s^{H-3/2}B(s){\rm d}s - T^{H-1/2}B(T) \\
& + (H-\tfrac{1}{2})\cdot\int_T^\infty \left[s^{H-3/2} - (s-\tau)^{H-3/2}\right] \cdot\big(B(s)-B(T)\big)\,{\rm d}s.
\end{align*}
Since Brownian motion is centred and it has independent increments, therefore $\tau$ is independent of $\{B(T+s)-B(T): s>0\}$ and the expected value of each term in the second and third line above is equal to $0$. This yields
\begin{align*}
\e\left(\tilde X_H^-(\tau)\right) = -\e\left((T-\tau)^{H-1/2}\big(B(\tau)-B(T)\big) - (H-\tfrac{1}{2})\cdot\int_\tau^T(s-\tau)^{H-3/2}(B(\tau)-B(s)){\rm d}s\right).
\end{align*}
After substituting $B(t)=Y(t)+at$ we find that the above equals
\begin{align*}
-\e\left((T-\tau)^{H-\tfrac{1}{2}}\big(Y(\tau)-Y(T)\big) - \frac{a(T-\tau)^{H+\tfrac{1}{2}}}{H+\tfrac{1}{2}} - (H-\tfrac{1}{2})\cdot\int_\tau^T(s-\tau)^{H-\tfrac{3}{2}}(Y(\tau)-Y(s)){\rm d}s\right).
\end{align*}
Now, let $\{\hat Y(t):t\in[0,T]\}$ be the time-reverse of process $Y$, that is $\hat Y(t):=Y(T-t)-Y(T)$ and let $\hat\tau:=\argmax\{t\in[0,T]:\hat Y(t)\}$ be its time of the supremum over $[0,T]$. Notice that we must have $\tau^*= T-\tau$. After substituting for $\hat Y$, we find that
\begin{align*}
\e\left(\tilde X_H^-(\tau)\right) = -\e\left(\hat\tau^{H-\tfrac{1}{2}}\hat Y(\tau) - \frac{a\hat\tau^{H+\tfrac{1}{2}}}{H+\tfrac{1}{2}} - (H-\tfrac{1}{2})\cdot\int_0^{\hat\tau}(\hat\tau-s)^{H-\tfrac{3}{2}}(\hat Y(\tau)-\hat Y(s)){\rm d}s\right).
\end{align*}
Finally, we notice that $\hat Y(\cdot) \eqd Y(\cdot;-a)$, that is $\hat Y$ is has the law of drifted Brownian motion with drift $-a$. Comparing the above with Eq.~\eqref{eq:Xplus_exp} for $\e\tilde X^+_H(\tau)$ concludes the proof.
\end{proof}

In light of the result in Lemma \ref{lem:J_antisymmetric} the function $\mathcal J^-_H(T,a)$ can be expressed using $\mathcal J^+_H(T,a)$, which justifies our notation $\mathcal J_H(T,a) := \mathcal J^+_H(T,a)$, cf. \eqref{def:J}. Before proving Proposition \ref{prop:Jvalue}, we need to establish a certain continuity property of the argmax functional of Brownian motion. In what follows, let $$\tau_{1/2}(T,a) := \argmax_{t\in[0,T]} \{B(t) - at\}$$ be the argmax functional of drifted Brownian motion, see also the definition \eqref{def:M_and_tau}.

\begin{lemma}\label{lem:tau_continuity}
It holds that
\begin{itemize}
\item[(i)] $\displaystyle \lim_{a\to a^*}\tau_{1/2}(T,a) = \tau_{1/2}(T,a^*)$ a.s. for any $T\in(0,\infty)$, $a^*\in\R$ and
\item[(ii)] $\displaystyle \lim_{T\to \infty} \tau_{1/2}(T,a) = \tau_{1/2}(\infty,a)$ a.s. for any $a>0$.
\end{itemize}
\end{lemma}

\begin{proof}
Let $Y(t;a) := B(t)-at$. It is easy to see that the trajectories $\{Y(t,a):t\in[0,T]\}$ converge uniformly to $\{Y(t,a^*):t\in[0,T]\}$, as $a\to a^*$, hence the argmax functionals also converge, see e.g. \cite[Lemma~2.9]{Seijo2011A}, which concludes the proof of item (i). Furthermore, since $\tau_{1/2}(\infty,a)$ is almost surely finite, then there must exist some (random) $T_0>0$ such that $\tau_{1/2}(T,a) = \tau_{1/2}(\infty,a)$ for all $T>T_0$, which implies item (ii).
\end{proof}

\begin{proof}[Proof of Proposition~\ref{prop:Jvalue}]
Let $H=\tfrac{1}{2}$ and observe that
\begin{align*}
\mathscr M_{1/2}(T,a) & =  \e \left(M^{(1,0)}_{1/2}(T,a)\right) \\
& = \e \left(B^{(1,0)}_{1/2}(\tau_{1/2}(T,a)) - a\tau_{1/2}(T,a)\right)\\
& = \frac{\e\left(\tilde X^+_{1/2}(\tau(T,a))\right)}{\sqrt{V^{(1,0)}_{1/2}}} - a\e\left(\tau_{1/2}(T,a)\right) \\
& = \mathcal J_{1/2}(T,a) - a\mathcal E_{1/2}(T,a),
\end{align*}
therefore $\mathcal J_{1/2}(T,a) = \mathscr M_{1/2}(T,a) + a\mathcal E_{1/2}(T,a)$, which agrees with Proposition~\ref{prop:tau_expectation} (note that the error function is a special case of the incomplete Gamma function, cf. \eqref{eq:erf_kummer}).

Till this end, let $H\in(0,\tfrac{1}{2})\cup(\tfrac{1}{2},1)$. For breviety denote $\tau := \tau(T,a)$ and let $Y(t) = B(t)-at$. Till this end, consider the case $T\in(0,\infty),a\neq0$. Recall that from \eqref{eq:Xplus_exp} we have
\begin{equation*}
\mathcal J_H(T,a) = \e\left( \tau^{H-\tfrac{1}{2}}Y(\tau) + \frac{a\tau^{H+\tfrac{1}{2}}}{H+\tfrac{1}{2}} -(H-\tfrac{1}{2}) \cdot\int_0^\tau (\tau-s)^{H-\tfrac{3}{2}}\cdot\big(Y(\tau)-Y(s)\big)\,{\rm d}s\right).
\end{equation*}
Now, we have
\begin{align*}
&\e\left(\int_0^\tau (\tau-s)^{H-\tfrac{3}{2}}\cdot\big(Y(\tau)-Y(s)\big)\,{\rm d}s\right) \\
& \qquad\qquad= \e\left(\e\left(\int_0^t (t-s)^{H-\tfrac{3}{2}}\cdot\big(Y(t)-Y(s)\big)\,{\rm d}s\mid \tau=t,Y(\tau)=y\right)\right),
\end{align*}
We now recognize that the above equals to $\e(I_H(\tau,Y(\tau))$, with $I_H(t,y)$ defined in Eq.~\eqref{def:I}. Using Lemma~\ref{lem:IH} we obtain
\begin{equation*}
\mathcal J_H(T,a) = \mathcal J_H^{(1)}(T,a) + \mathcal J_H^{(2)}(T,a),
\end{equation*}
where
\begin{align*}
\mathcal J_H^{(1)}(T,a) & := \frac{1}{H+\tfrac{1}{2}}\cdot \e \left(\tau^{H-\tfrac{1}{2}}\left(Y(\tau) + a\tau - \frac{\tau}{Y(\tau)}\right)\right) \\
\mathcal J_H^{(2)}(T,a) & := \frac{1}{H+\tfrac{1}{2}} \cdot
\e\left(\frac{\tau^{H+\tfrac{1}{2}}}{Y(\tau)} \cdot \frac{\Gamma(H)}{\sqrt{\pi}}U\left(H-\frac{1}{2},\frac{1}{2},\frac{Y^2(\tau)}{2\tau}\right)\right).
\end{align*}
The joint density of the pair $(\tau,Y(\tau))$ is well-known, see \eqref{eq:density_T} in Section \ref{s:joint_density} therefore both functions $\mathcal J_H^{(1)}(T,a)$ and $\mathcal J_H^{(2)}(T,a)$ can be written as definite integrals and calculated. In fact, we have
\begin{equation}\label{eq:proof_calculation_Jvalue}
\mathcal J_H^{(1)}(T,a) = 0, \qquad \mathcal J_H^{(2)}(T,a) = \frac{2^{H}}{\sqrt{2\pi}(H+\tfrac{1}{2})} \cdot |a|^{-2H}\gamma(H,\tfrac{a^2T}{2}).
\end{equation}
The derivation of Eq.~\eqref{eq:proof_calculation_Jvalue} is purely calculational and is provided in Appendix \ref{appendix:calculations}. This ends the proof in case $T\in(0,\infty)$, $a\neq 0$.

In order to derive the formula for $\mathcal J_H(T,a)$ in the remaining two cases (i.e. $T\in(0,\infty)$, $a=0$ and $T=\infty$, $a>0$), we could redo the calculations in Eq.~\eqref{eq:proof_calculation_Jvalue} with appropriate density functions for the pair $(\tau,Y(\tau))$, see \eqref{eq:density_infty} and \eqref{eq:density_T}. However, it is not necessary, as it suffices to show that the function $\mathcal J_H(T,a)$ is continuous at $(\infty,a)$ and $(T,0)$.
 
Let $T\in(0,\infty)$, $a=0$. Using the fact that $\gamma(s,x) \sim (sx^s)^{-1}$, as $x\downarrow0$, it can be seen that
\begin{align*}
\lim_{a\to0}\mathcal J_H(T,a) = \frac{T^{H}}{\sqrt{2\pi}H(H+\tfrac{1}{2})}.
\end{align*}
Showing that $\lim_{a\to0}\mathcal J_H(T,a)  = \mathcal J_H(T,0)$ would therefore conclude the proof in this case. By the definition, we have $\mathcal J_H(T,a) = \sqrt{V(H)} \cdot \e (B^{(1,0)}_H(\tau(T,a)))$. Since $V(H)$ is continuous at $H=\tfrac{1}{2}$, therefore it suffices to show
\begin{equation}\label{eq:to_show_prop_a_zero}
\e (B^{(1,0)}_H(\tau(T,a))) \to \e (B^{(1,0)}_H(\tau(T,0))),
\end{equation}
as $a\to 0$. Using Proposition \ref{prop:continuous_modification_pm} and Lemma \ref{lem:tau_continuity}(i) we obtain
\begin{align*}
\lim_{a\to 0} B_H^{(1,0)}(\tau(T,a)) = B_H^{(1,0)}(\tau(T,0)) \quad \text{a.s.}
\end{align*}
Moreover, for any $\ep>0$ and all $a\in(-\ep,\ep)$ we have \[B_H^{(1,0)}(\tau(T,a)) \leq \sup_{t\in[0,T]}B^{(1,0)}_H(t) + \ep T,\]
which has finite expectation. Therefore, by Lebesgue dominated convergence we can conclude that the limit \eqref{eq:to_show_prop_a_zero} holds, which ends the proof in this case.

Let $T=\infty$, $a>0$. It is easy to see that 
\begin{align*}
\lim_{T\to\infty}\mathcal J_H(T,a) = \frac{2^{H}\Gamma(H)}{\sqrt{2\pi}(H+\tfrac{1}{2})} \cdot |a|^{-2H}, \quad \text{if} \ a>0.
\end{align*}
Analogously to the proof of the previous case, it suffices to show
\begin{equation}\label{eq:to_show_prop_T_infty}
\e (B^{(1,0)}_H(\tau(T,a))) \to \e (B^{(1,0)}_H(\tau(\infty,a))),
\end{equation}
as $T\to\infty$. Using Proposition \ref{prop:continuous_modification_pm} and Lemma \ref{lem:tau_continuity}(ii) we find that
\begin{align*}
\lim_{T\to \infty} B_H^{(1,0)}(\tau(T,a)) = B_H^{(1,0)}(\tau(\infty,a)) \quad \text{a.s.}
\end{align*}
Moreover, since the mapping $T\mapsto\tau(T,a)$ is nondecreasing, it holds that
\[B_H^{(1,0)}(\tau(T,a)) \leq M^{(1,0)}_H(\infty,a)+ a\tau(\infty,a),\]
and the right-hand-side above has a finite expectation. Using Lebesgue dominated convergence we conclude that the limit \eqref{eq:to_show_prop_T_infty} holds, which ends the proof.
\end{proof}

\begin{proof}[Proof of Proposition~\ref{prop:mH}]
From the definition of $m^\apm_H(T,a)$ in \eqref{def:mHapm}, 
\begin{align*}
m^\apm_H(T,a) & := \e \left\{\tilde B^\apm_H(\tau(T,a)) - a\tau(T,a)\right\}\\
& = \frac{\ap \mathcal J_H^+(T,a)-\am \mathcal J_H^-(T,a)}{\sqrt{V^\apm_H}} - a\mathcal E_{1/2}(T,a).
\end{align*}
In light of Lemma \ref{lem:J_antisymmetric}, for $T\in(0,\infty)$ we have
\begin{equation}\label{eq:m_formula}
m^\apm_H(T,a) = \frac{\ap-\am }{\sqrt{V^\apm_H}}\cdot \mathcal J_H(T,a) - a\mathcal E_{1/2}(T,a).
\end{equation}
We will now show that Eq.~\eqref{eq:m_formula} holds also in case $T=\infty$, $a>0$. Using Proposition \ref{prop:continuous_modification_pm} and Lemma~\ref{lem:tau_continuity}(ii) we find that
$$\tilde B^\apm_H(\tau(T,a)) - a\tau(T,a) \to \tilde B^\apm_H(\tau(\infty,a)) - a\tau(\infty,a) \ \ \text{a.s.}, \quad T\to\infty.$$
Now we have the bound $\tilde B^\apm_H(\tau(T,a)) - a\tau(T,a) \leq M^\apm_{H}(\infty,a)$, which is integrable, hence we can conclude that
\begin{align*}
m^\apm_H(\infty,a) = \lim_{T\to\infty} m^\apm_H(T,a).
\end{align*}
Furthermore, from Proposition \ref{prop:Jvalue} and Proposition \ref{prop:tau_expectation} it is clear that $\mathcal J_H(T,a) \to \mathcal J_H(\infty,a)$ and $\mathcal E_{1/2}(T,a) \to \mathcal E_{1/2}(\infty,a)$ as $T\to\infty$, respectively. Hence,
\begin{align*}
\lim_{T\to\infty} m^\apm_H(T,a) = \frac{\ap-\am }{\sqrt{V^\apm_H}}\cdot \mathcal J_H(T,a) - a\mathcal E_{1/2}(T,a),
\end{align*}
which concludes the proof that Eq~\eqref{eq:m_formula} holds for all admissible pairs $(T,a)$.

Finally, since $J_H(T,a)>0$ (see Proposition \ref{prop:Jvalue}), it is easy to see that \eqref{eq:m_formula} is maximized whenever $(\ap-\am)/\sqrt{V^\apm_H}$ is maximized. It is a straighforward to show that the maximum is attained at $\apm=(1,-1)$, which concludes the proof.
\end{proof}

%---------------------------------------------------------
%---------------------------------------------------------
%---------------------------------------------------------

\appendix

%---------------------------------------------------------
%---------------------------------------------------------
%---------------------------------------------------------

\section{Special functions}\label{appendix:special_functions}
All of definitions, formulae and relations from this section can be found in \cite{abramowitz1988handbook}.
\subsubsection*{Confluent hypergeometric functions}
For any $a,z\in\R$ and $b\in\R\setminus\{0,-1,-2,\ldots\}$ we define Kummer's (confluent hypergeometric) function
\begin{equation}\label{def:kummer}
\kummer(a,b,z) := \sum_{n=0}^\infty \frac{a^{(n)}z^n}{b^{(n)}n!},
\end{equation}
where $a^{(n)}$ is the rising factorial, i.e.
\begin{align*}
a^{(0)} := 1, \quad \text{and} \quad a^{(n)}:=a(a+1)\cdots(a+n-1) \ \text{for} \ n\in\N.
\end{align*}
Similarly, for any $a,z\in\R$ and $b\in\R\setminus\{0,-1,-2,\ldots\}$ we define Tricomi's (confluent hypergeometric) function, i.e.
\begin{equation}\label{def:tricomi}
U(a,b,z) = \frac{\Gamma(1-b)}{\Gamma(a+1-b)}{}_1F_1(a,b,z) + \frac{\Gamma(b-1)}{\Gamma(a)}{}_1F_1(a+1-b,2-b,z).
\end{equation}
When $b>a>0$, then Kummer's function can be represented as an integral
\begin{equation}
\label{eq:kummer_integral_representation}
\kummer(a,b,z) = \frac{\Gamma(b)}{\Gamma(a)\Gamma(b-a)}\int_0^1 e^{zt}t^{a-1}(1-t)^{b-a-1}{\rm d}t
\end{equation}
and similarly, for $a>0$, $z>0$ Tricomi's function can be represented as an integral
\begin{equation}
\label{eq:tricomi_integral_representation}
U(a,b,z) = \frac{1}{\Gamma(a)}\int_0^\infty e^{-zt}t^{a-1}(1+t)^{b-a-1}{\rm d}t.
\end{equation}
Moreover, the following Kummer's transformations hold
\begin{align}
\label{def:kummer_transformation}\kummer(a,b,z) & = e^{z}\kummer(b-a,b,-z), \\
\label{def:tricomi_transformation}U(a,b,z) & = z^{1-b}U(1+a-b,2-b,z).
\end{align}
and the following recurrence relations hold:
\begin{align}
\label{align:Urecurrence1} z U(a,b+1,z) & = (b-a)U(a,b,z)+U(a-1,b,z) \\
\label{align:Urecurrence2} a U(a+1,b,z) & = U(a,b,z) - U(a,b-1,z).
\end{align}
In this manuscript, we often use the following integral equality. Let $c>\gamma>0$ and $u>0$, then
\begin{equation}\label{key_integral}
\int_0^1 x^{\gamma-1}(1-x)^{c-\gamma-1}\kummer(a,\gamma,ux){\rm d}x = \frac{\Gamma(\gamma)\Gamma(c-\gamma)}{\Gamma(c)}\kummer(a,c,u),
\end{equation}
which can be verified using \cite[7.613-1]{integral_table}.

\subsubsection*{Incomplete Gamma function}
For any $\alpha>0$, $z>0$ we define the \emph{upper} and \emph{lower} incomplete Gamma functions respectively
\begin{equation}\label{def:incomplete_gamma}
\Gamma(\alpha,z) := \int_z^\infty t^{\alpha-1}e^{-t}{\rm d}t, \quad \gamma(\alpha,z) := \int_0^z t^{\alpha-1}e^{-t}{\rm d}t,
\end{equation}
so that $\Gamma(\alpha,z)+\gamma(\alpha,z)=\Gamma(\alpha)$, where $\Gamma(\cdot)$ the standard Gamma function $\Gamma(\alpha):=\int_0^\infty t^{\alpha-1}e^{-t}{\rm d}t$. Using integration by parts we obtain the following useful recurrence relation
\begin{equation}\label{eq:gamma_recurrence}
\gamma(\alpha+1,z) = \alpha\gamma(\alpha,z) - z^\alpha e^{-z}.
\end{equation}
Notice that as $z\to\infty$, the above is reduced to the well-known recurrence relation for the Gamma function, that is $\Gamma(\alpha+1) = \alpha\Gamma(\alpha)$. Finally, we note that $\gamma(\alpha,z)$ can be expressed in terms of the confluent hypergeometric function:
\begin{equation}\label{eq:lower_gamma_kummer}
\gamma(\alpha,z)= \alpha^{-1}z^\alpha e^{-z}\kummer(1,\alpha+1,z).
\end{equation}

\subsubsection*{Error function}
For $z\in\R$ we define the error function and complementary error function respectively:
\begin{equation}\label{def:erf}
\erf(z) := \frac{2}{\sqrt{\pi}} \int_0^z e^{-t^2}{\rm d}t, \quad \erfc(z) := 1-{\rm erf}(z).
\end{equation}
The error function can be expressed in terms of the incomplete Gamma function, and as such, in light of \eqref{eq:lower_gamma_kummer}, also in terms of the hypergeometric function:
\begin{equation}\label{eq:erf_kummer}
\erf(z) = \frac{\sgn(z)}{\sqrt{\pi}}\gamma(\tfrac{1}{2},z^2) = \frac{2ze^{-z^2}}{\sqrt{\pi}}\kummer(1,\tfrac{3}{2},z^2)
\end{equation}

%---------------------------------------------------------
%---------------------------------------------------------
%---------------------------------------------------------

\section{Calculations}\label{appendix:calculations}

\begin{proof}[Derivation of the formula for $\mathcal E_{1/2}(T,a)$ in Proposition \ref{prop:tau_expectation}]
The derivation in cases $T=\infty$, $a>0$ and $T\in(0,\infty)$, $a=0$ is straightforward using Eq.~\eqref{eq:density_infty} and Eq.~\eqref{eq:density_T} respectively. Till this end assume that $a\neq 0$ and $T\in(0,\infty)$. We have
\begin{align*}
\p(\tau(T,a)\in {\rm d}t) = \left(\frac{e^{-a^2t/2}}{\sqrt{\pi t}} - \frac{a}{\sqrt{2}}\erfc\left(\frac{a\sqrt{t}}{\sqrt{2}}\right)\right)\left(\frac{e^{-a^2(T-t)/2}}{\sqrt{\pi (T-t)}} + \frac{a}{\sqrt{2}}\erfc\left(-\frac{a\sqrt{T-t}}{\sqrt{2}}\right)\right),
\end{align*}
see e.g. 2.1.12.4 in \cite{borodin2002handbook}. We now have
\begin{align*}
\mathcal E_{1/2}(T,a) & = \int_0^T t\p(\tau(T,a)\in {\rm d}t) {\rm d}t \\
& = T\int_0^1 t\left(\frac{e^{-u^2t}}{\sqrt{\pi t}} - u\erfc(u\sqrt{t})\right)\left(\frac{e^{-u^2(1-t)}}{\sqrt{\pi(1-t)}} + u\erfc(-u\sqrt{1-t})\right){\rm d}t,
\end{align*}
where we substituted $t = tT$ and put $u := a\sqrt{T/2}$. Now, we have
\begin{equation}\label{eq:final_formula_E12}
\mathcal E_{1/2}(T,a) = T \cdot \left(\frac{e^{-u^2}}{\pi}\cdot J_1 + \frac{u}{\sqrt{\pi}}\cdot\left(J_2(u)-J_3(u)\right) - u^2 J_4(u) \right),
\end{equation}
where
\begin{align*}
& J_1 := \int_0^1 \sqrt{\frac{t}{1-t}}{\rm d} t = \frac{\pi}{2}, \quad J_2(u) := \int_0^1 \sqrt{t}e^{-u^2 t}\erfc(-u\sqrt{1-t}){\rm d}t,\\
& J_3(u) := \int_0^1 \frac{t}{\sqrt{1-t}}e^{-u^2(1-t)}\erfc(u\sqrt{t}){\rm d}t, \quad J_4(u) := \int_0^1 t\erfc(u\sqrt{t})\erfc(-u\sqrt{1-t}){\rm d}t.
\end{align*}
Using the fact that $\erfc(-z) = 2-\erfc(z)$ and applying substitution $t = 1-t$, we obtain
\begin{align*}
J_2(u) & = 2\int_0^1\sqrt{t}e^{-u^2t}{\rm d}t - \int_0^1 \sqrt{1-t}e^{-u^2(1-t)}\erfc(u\sqrt{t}){\rm d}t.
\end{align*}
In the following, let
\begin{align*}
J_5(u) := \int_0^1\sqrt{t}e^{-u^2t}{\rm d}t = \frac{\sqrt{\pi}\erf(u) - 2ue^{-u^2}}{2u^3}, \quad
J_6(u) := \int_0^1\frac{e^{-u^2(1-t)}}{\sqrt{1-t}}\erfc(u\sqrt{t}){\rm d}t,
\end{align*}
where the first integral can be calculated using substitution $t = x^2$. Then
\begin{align*}
J_2(u) & = 2J_5(u) - \int_0^1\frac{1-t}{\sqrt{1-t}}e^{-u^2(1-t)}\erfc(u\sqrt{t}){\rm d}t \\
& = 2J_5(u) + J_3(u) - J_6(u).
\end{align*}
Now, applying substitution $t = x^2/u^2$ and formula \cite[4.3.20]{ng1969table} we obtain
\begin{equation*}
J_6(u) = u^{-1}e^{-u^2}\int_0^u \frac{2x e^{x^2}}{\sqrt{u^2-x^2}}\erfc(u x){\rm d}t = \frac{\sqrt{\pi}}{u}\cdot\left(e^{-u^2}-\erfc(u)\right).
\end{equation*}
Further, using integration by parts we obtain
\begin{align*}
J_3(u) & = \frac{-2ue^{-u^2(1-t)}\sqrt{1-t} + \sqrt{\pi}(1-2u^2)\erf(u\sqrt{1-t})}{2u^3} \cdot \erfc(u\sqrt{t})\bigg|_0^1\\
& \quad + \int_0^1\frac{-2ue^{-u^2(1-t)}\sqrt{1-t} + \sqrt{\pi}(1-2u^2)\erf(u\sqrt{1-t})}{2u^3} \cdot \frac{ue^{-u^2t}}{\sqrt{\pi t}}{\rm d}t.
\end{align*}
In the following let
\begin{align*}
J_7(u) := \int_0^1 \frac{e^{-u^2t}}{\sqrt{t}}{\rm d}t = \frac{\sqrt{\pi}\erf(u)}{u},
\end{align*}
which can be easily calculated using substitution $t = x^2$. Then
\begin{align*}
J_3(u) & = \frac{2ue^{-u^2} - \sqrt{\pi}(1-2u^2)\erf(u)}{2u^3} - \frac{e^{-u^2}}{u\sqrt{\pi}}\cdot J_1 + \left(\tfrac{1}{2u^2}-1\right)(J_7(u)-J_6(u))
\end{align*}
and therefore
\begin{equation}\label{eq:tauexp_J2J3}
\begin{split}
J_3(u) & = \frac{e^{-u^2}}{2u^3}\left(2u - \sqrt{\pi}(1-u^2) - e^{u^2}\sqrt{\pi}(2u^2-1)\erfc(u)\right)\\
J_2(u) & = \frac{e^{-u^2}}{2u^3}\left(-2u - \sqrt{\pi}(1+u^2) + e^{u^2}\sqrt{\pi}(1+\erf(u))\right).
\end{split}
\end{equation}
In the following, let
\begin{align*}
J_8(u) & :=\int_0^1 t\erfc(u\sqrt{t}){\rm d}t = \frac{1}{8}\left(4 -\frac{2e^{-u^2}(3+2u^2)}{\sqrt{\pi}u^3} + \left(\frac{3}{u^4}-4\right)\erf(u)\right),\\
J_9(u) & := \int_0^1 t\erfc(u\sqrt{t})\erfc(-u\sqrt{1-t}){\rm d}t,
\end{align*}
where the first integral was calculated using integration by parts. Using the identity $\erfc(-z)=2-\erfc(z)$ we find that
\begin{align*}
J_4(u) & = 2J_8(u) - J_9(u).
\end{align*}
We will now find the value of the integral $J_9(u)$. See that
After applying substitution $t=1-t$ we obtain
\begin{align*}
2J_9(u) & = \int_0^1 t\erfc(u\sqrt{t})\erfc(u\sqrt{1-t}){\rm d}t +\int_0^1 (1-t)\erfc(u\sqrt{t})\erfc(u\sqrt{1-t}){\rm d}t \\
& = \int_0^1 \erfc(u\sqrt{t})\erfc(u\sqrt{1-t}){\rm d}t.
\end{align*}
After integration by parts we find that
\begin{align*}
2J_9(u) & = \left(\frac{\erf(u\sqrt{t})}{2u^2} + t\erfc(u\sqrt{t}) - \frac{\sqrt{t}e^{-u^2t}}{\sqrt{\pi}u}\right)\cdot \erfc(u\sqrt{1-t})\Bigg|_0^1 \\
& - \int_0^1\left(\frac{\erf(u\sqrt{t})}{2u^2} + t\erfc(u\sqrt{t}) - \frac{\sqrt{t}e^{-u^2t}}{\sqrt{\pi}u}\right) \cdot \frac{ue^{-u^2(1-t)}}{\sqrt{\pi(1-t)}}{\rm d}t,
\end{align*}
therefore
\begin{align*}
2J_9(u) & = \left(\frac{\erf(u)}{2u^2} + \erfc(u) - \frac{e^{-u^2}}{\sqrt{\pi}u}\right)- \frac{1}{2\sqrt{\pi}u}\left(J_7(u) - J_6(u) \right) -\frac{u }{\sqrt{\pi}}\cdot J_3(u) + \frac{e^{-u^2} }{\pi} \cdot J_1.
\end{align*}
After simple algebraic manipulations this yields
\begin{equation}\label{eq:tauexp_J4}
J_4(u) = \frac{3 - e^{-u^2}\left(2u^2+\frac{6u}{\sqrt{\pi}}\right) + (2u^2-3)\erfc(u)}{4u^4}.
\end{equation}
Finally, after plugging in the expressions for $J_2(u)$, $J_3(u)$ and $J_4(u)$ calculated in \eqref{eq:tauexp_J2J3} and \eqref{eq:tauexp_J4} into Eq.~\eqref{eq:final_formula_E12} we obtain
\begin{equation*}
\mathcal E_{1/2}(T,a) = \frac{T}{2}\left(1 +\left(\frac{1}{2u^2}-1\right)\erf(u) - \frac{e^{-u^2}}{\sqrt{\pi}u}\right),
\end{equation*}
which concludes the proof.
\end{proof}

\begin{proof}[Continuation of the proof of Lemma \ref{lem:IH}]
First we show that
\begin{equation}\label{eq:IH_integral}
I_H(t,y) = \frac{t^{H+1/2}}{\sqrt{2\pi}y} \int_0^\infty\int_0^\infty\frac{q^{H-3/2}x^2}{(1+q)^{H+3/2}}\cdot\left[e^{-\tfrac{1}{2}\left(x-y\sqrt{\frac{q(s)}{t}}\right)^2} - e^{-\tfrac{1}{2}\left(x+y\sqrt{\frac{q(s)}{t}}\right)^2}\right]\,{\rm d}x{\rm d}q,
\end{equation}
where $q(s) = s/(1-s)$, which generalizes \cite[Proposition~1]{bisewski2021derivatives} in case $H=\tfrac{1}{2}$. From \eqref{eq:IH_integral} we have
\begin{align*}
I_H(t,y) & = \frac{t^{3/2}}{y} \int_0^t\int_0^\infty\frac{x^2}{s^{3-H}\sqrt{2\pi(t-s)}}e^{-\frac{x^2}{2s}+\frac{y^2}{2t}}\cdot\left[e^{-\frac{(x-y)^2}{2(t-s)}} - e^{-\frac{(x+y)^2}{2(t-s)}}\right]\,{\rm d}x{\rm d}s \\
& = \frac{t^{3/2}}{\sqrt{2\pi}y} \int_0^t\int_0^\infty\frac{x^2}{s^{3-H}\sqrt{t-s}}\cdot\left[e^{-\tfrac{1}{2}\big(\tfrac{x}{\sigma_0}-y\mu_0\big)^2} - e^{-\tfrac{1}{2}\big(\tfrac{x}{\sigma_0}+y\mu_0\big)^2}\right]\,{\rm d}x{\rm d}s,
\end{align*}
where $\mu_0^2 := \frac{s}{(t-s)t}$, and $\sigma_0^2 := \frac{(t-s)s}{t}$. We now apply substitution $x := \sigma_0 x$, then $s := ts$, which gives us
\begin{equation*}
I_H(t,y) = \frac{t^{H+1/2}}{\sqrt{2\pi}y} \int_0^1\int_0^\infty\frac{(1-s)x^2}{s^{3/2-H}}\cdot\left[e^{-\tfrac{1}{2}\left(x-y\sqrt{\frac{q(s)}{t}}\right)^2} - e^{-\tfrac{1}{2}\left(x+y\sqrt{\frac{q(s)}{t}}\right)^2}\right]\,{\rm d}x{\rm d}s,
\end{equation*}
which after substitution $q = s/(1-s)$ yields \eqref{eq:IH_integral}. from \cite[Eq.~(54)]{bisewski2021derivatives} we find that for any $b\in\R$:
\begin{equation*}\label{eq:integral_b}
\int_0^\infty x^2\left(e^{-\frac{(x-b)^2}{2}}-e^{-\frac{(x+b)^2}{2}}\right){\rm d}x = 2be^{-b^2/2} + (1+b^2)\sqrt{2\pi}\cdot \erf(b/\sqrt{2}).
\end{equation*}
Therefore, 
\begin{equation}\label{eq:IH_sum}
I_H(t,y) = \frac{t^{H+1/2}}{\sqrt{\pi}y}(J_1(H,u)+J_2(H,u)+J_3(H,u)),
\end{equation}
where we put $u:=\frac{y}{\sqrt{2t}}$ and
\begin{align*}
J_1(H,u) & := \int_0^\infty\frac{q^{H-3/2}}{(1+q)^{H+3/2}}\cdot 2u\sqrt{q}e^{-uq}{\rm d}q,\\
J_2(H,u) & := \int_0^\infty\frac{q^{H-3/2}}{(1+q)^{H+3/2}}\cdot \sqrt{\pi}\erf(u\sqrt{q}){\rm d}q \\
J_3(H,u) & :=  \int_0^\infty\frac{q^{H-3/2}}{(1+q)^{H+3/2}} \cdot 2u^2q^2\sqrt{\pi}\erf(u\sqrt{q}){\rm d}q
\end{align*}
Using \eqref{eq:tricomi_integral_representation} we find that
\begin{align*}
J_1(H,u) = 2u\Gamma(H)U(H,-\tfrac{1}{2},u^2).
\end{align*}
Integration by parts yields
\begin{align*}
J_2(H,u) = \frac{\sqrt{\pi}(2+4H+4q)q^{H-\tfrac{1}{2}}}{(4H^2-1)(1+q)^{H+1/2}} \cdot \erf(u\sqrt{q})\bigg|_0^\infty -\int_0^\infty \frac{\sqrt{\pi}(2+4H+4q)q^{H-\tfrac{1}{2}}}{(4H^2-1)(1+q)^{H+1/2}}\cdot \frac{ue^{-u^2q}}{\sqrt{\pi q}}{\rm d}q
\end{align*}
and after applying \eqref{eq:tricomi_integral_representation} we find that
\begin{align*}
J_2(H,u) = \frac{\sqrt{\pi}}{H^2-\tfrac{1}{4}} - \frac{u\Gamma(H)U(H,\tfrac{1}{2},u^2)}{H-\tfrac{1}{2}} - \frac{u\Gamma(H)U(1+H,\tfrac{3}{2},u^2)}{H^2-\tfrac{1}{4}}.
\end{align*}
The integral $J_3(H,u)$ is computed analogously to $J_2(H,u)$ and it equals
\begin{align*}
J_3(u,H) = \frac{2\sqrt{\pi}u^2}{H+\tfrac{1}{2}} - \frac{2u^3\Gamma(1+H)U(1+H,\tfrac{3}{2},u^2)}{H+\tfrac{1}{2}}.
\end{align*}
Now, after simple algebraic manipulations we obtain
\begin{align*}
J_1(H,u)+J_2(H,u)+J_3(H,u) = \frac{\sqrt{\pi}}{H^2-\tfrac{1}{4}}\left(1 - \frac{\Gamma(H)}{\sqrt{\pi}}J_4(H,u)\right) + \frac{2\sqrt{\pi}u^2}{H+\tfrac{1}{2}},
\end{align*}
where
\begin{align*}
J_4(H,u) & = -2u(H^2-\tfrac{1}{4})U(H,-\tfrac{1}{2},u^2) + u(H+\tfrac{1}{2})U(H,\tfrac{1}{2},u^2) \\
& \quad + uHU(H+1,\tfrac{3}{2},u^2) + 2u^3(H-\tfrac{1}{2})HU(H+1,\tfrac{3}{2},u^2).
\end{align*}
Applying the relation \eqref{align:Urecurrence1} to the last term above we obtain
\begin{align*}
J_4(H,u) & = -2u(H^2-\tfrac{1}{4})\Big(U(H,-\tfrac{1}{2},u^2) + HU(H+1,\tfrac{1}{2},u^2\Big) + u(H+\tfrac{1}{2})U(H,\tfrac{1}{2},u^2)\\
& \quad + uHU(H+1,\tfrac{3}{2},u^2) +  2u(H-\tfrac{1}{2})HU(H,\tfrac{1}{2},u^2).
\end{align*}
Applying the relation \eqref{align:Urecurrence2} to the first and third terms above we obtain
\begin{align*}
J_4(H,u) & = uU(H,\tfrac{3}{2},u^2) = U(H-\tfrac{1}{2},\tfrac{1}{2},u^2),
\end{align*}
where in the last equality we applied the Kummer's transformation \eqref{def:tricomi_transformation}. This concludes the proof.
\end{proof}

\begin{lemma}\label{lemma:U_integral}
If $H\in(0,\tfrac{1}{2})\cup(\tfrac{1}{2},1)$ and $u\in\R$, then
\begin{align*}
& \int_0^\infty U(H-\tfrac{1}{2},\tfrac{1}{2},z^2)e^{-(z+u)^2}{\rm d}z \\
& \qquad= \frac{\sqrt{\pi}e^{-u^2}}{2\Gamma(H+\tfrac{1}{2})} + \frac{\sqrt{\pi}|u|^{1-2H}}{2}\left(\frac{\gamma(H+\tfrac{1}{2},u^2)}{\Gamma(H+\tfrac{1}{2})}-\sgn(u)\frac{\gamma(H,u^2)}{\Gamma(H)}\right).
\end{align*}
\iffalse
\begin{align*}
\int_0^\infty U(H-\tfrac{1}{2},\tfrac{1}{2},y^2)e^{-(y+u)^2}{\rm d}y= \frac{\sqrt{\pi}}{2}\cdot u^{1-2 H}\left( \frac{\gamma(H-\tfrac{1}{2},u^2)}{\Gamma(H-\tfrac{1}{2})}-\frac{\gamma(H,u^2)}{\Gamma(H)}\right)
\end{align*}
For $u<0$ we have
\begin{align*}
\int_0^\infty U(H-\tfrac{1}{2},\tfrac{1}{2},y^2)e^{-(y+u)^2}{\rm d}y= \frac{\sqrt{\pi}}{2}\cdot |u|^{1-2H}\left(\frac{\gamma(H-\tfrac{1}{2},u^2)}{\Gamma(H-\tfrac{1}{2})}+\frac{\gamma(H,u^2)}{\Gamma(H)}\right)
\end{align*}
\begin{align*}
\int_0^\infty U(H-\tfrac{1}{2},\tfrac{1}{2},y^2)e^{-(y+u)^2}{\rm d}y =
\begin{cases}
\frac{\sqrt{\pi}}{2}\cdot u^{1-2 H}\left( \frac{\gamma(H-\tfrac{1}{2},u^2)}{\Gamma(H-\tfrac{1}{2})}-\frac{\gamma(H,u^2)}{\Gamma(H)}\right), & u\geq 0\\
\frac{\sqrt{\pi}}{2}\cdot |u|^{1-2H}\left(\frac{\gamma(H-\tfrac{1}{2},u^2)}{\Gamma(H-\tfrac{1}{2})}+\frac{\gamma(H,u^2)}{\Gamma(H)}\right), & u<0
\end{cases}
\end{align*}
\fi
\end{lemma}
\begin{proof}
Using \eqref{def:kummer_transformation} we find that $U(H-\tfrac{1}{2},\tfrac{1}{2},z^2) = zU(H,\tfrac{3}{2},z^2)$ therefore, using \eqref{eq:tricomi_integral_representation} we have
\begin{align*}
A(u,H) & := \int_0^\infty U(H-\tfrac{1}{2},\tfrac{1}{2},z^2)e^{-(z+u)^2}{\rm d}z\\
 & =\frac{1}{\Gamma(H)}\int_0^\infty t^{H-1}(1+t)^{1/2-H}\int_0^\infty ze^{-(z+u)^2}e^{-z^2t}{\rm d}z{\rm d}t \\
& = \frac{e^{-u^2}}{\Gamma(H)}\int_0^\infty t^{H-1}(1+t)^{1/2-H}\int_0^\infty ze^{-z^2(1+t)-2zu}{\rm d}z{\rm d}t\\
& = \frac{e^{-u^2}}{\Gamma(H)}\int_0^\infty t^{H-1}(1+t)^{1/2-H} \cdot \frac{\sqrt{1+t} - \sqrt{\pi}ue^{u^2/(1+t)}\erfc(u/\sqrt{1+t})}{2(1+t)^{3/2}}{\rm d}t
\end{align*}
After substitution $s = \frac{1}{1+t}$ we obtain
\begin{align*}
A(u,H)  =\frac{e^{-u^2}}{2\Gamma(H)} \left(J_1(u,H) - J_2(u,H) + J_3(u,H)\right),
\end{align*}
where
\begin{align*}
J_1(u,H) & := \int_0^1 s^{-1/2}(1-s)^{H-1}{\rm d}s,\\
J_2(u,H) & := \sqrt{\pi}u\int_0^1 (1-s)^{H-1}e^{u^2s}{\rm d}s,\\
J_3(u,H) &:= \sqrt{\pi}u\int_0^1 (1-s)^{H-1}e^{u^2s}\erf(u\sqrt{z}){\rm d}s.
\end{align*}
Now, $J_1(u,H)$ can be easily found from the definition of the Beta function, while $J_2(u,H)$ can be found using \eqref{eq:kummer_integral_representation}, hence
\begin{align*}
J_1(u,H) = \frac{\sqrt{\pi}\Gamma(H)}{\Gamma(H+\tfrac{1}{2})}, \quad J_2(u,H) = \frac{\sqrt{\pi}u\Gamma(H)}{\Gamma(H+1)}\kummer(1,H+1,u^2).
\end{align*}
The integral $J_3(u,H)$ is now computed using the error function representation from \eqref{eq:erf_kummer} and the result in \eqref{key_integral}, i.e.
\begin{align*}
J_3(u,H) & = 2u^2\int_0^1 s^{1/2}(1-s)^{H-1}\kummer(\tfrac{1}{2},\tfrac{3}{2},u^2s){\rm d}s \\
& = 2u^2 \cdot \frac{\Gamma(\tfrac{3}{2})\Gamma(H)}{\Gamma(H+\tfrac{3}{2})}\kummer(1,H+\tfrac{3}{2},u^2).
\end{align*}
Finally, using \eqref{eq:lower_gamma_kummer} we obtain
\begin{align*}
A(u,H) & = \frac{\sqrt{\pi}}{2} \left(\frac{e^{-u^2}}{\Gamma(H+\tfrac{1}{2})} - \frac{ue^{-u^2}}{\Gamma(H+1)}\kummer(1,H+1,u^2) + \frac{u^2e^{-u^2}}{\Gamma(H+\tfrac{3}{2})}\kummer(1,H+\tfrac{3}{2},u^2)\right) \\
& = \frac{\sqrt{\pi}}{2} \left(\frac{e^{-u^2}}{\Gamma(H+\tfrac{1}{2})} - \frac{u|u|^{-2H}\gamma(H,u^2)}{\Gamma(H)} + \frac{|u|^{-2H+1}\gamma(H+\tfrac{1}{2},u^2)}{\Gamma(H+\tfrac{1}{2})}\right),
\end{align*}
which ends the proof.
\end{proof}

\begin{proof}[Proof of Eq.~\eqref{eq:proof_calculation_Jvalue}] Recall that $H\in(0,\tfrac{1}{2})\cup(\tfrac{1}{2},1)$, $a\neq0$ and $T\in(0,\infty)$. We will first show that
\begin{equation}\label{eq:J1_zero}
\mathcal J_H^{(1)}(T,a) := \e \left(\tau^{H-\tfrac{1}{2}}\left(Y(\tau) + a\tau - \frac{\tau}{Y(\tau)}\right)\right) = 0.
\end{equation}
The joint density $p(t,y;a,T)$ of the pair $(\tau,Y(\tau))$ is well-known, see \eqref{eq:density_T} in Section \ref{s:joint_density}. We have $p(t,y;T,a) = p_1(t,y;T,a)p_2(t;T,a)$, where
\begin{align*}
& p_1(t,y;T,a) := yt^{-3/2}\exp\left\{-\tfrac{(y+ta)^2}{2t}\right\}\\
& p_2(t;T,a) := \frac{1}{\pi\sqrt{T-t}} \cdot\Big(e^{-a^2(T-t)/2} + a\sqrt{\tfrac{\pi(T-t)}{2}}~\erfc\Big(-a\sqrt{\tfrac{T-t}{2}}\Big)\Big),
\end{align*}
hence
\begin{align*}
\mathcal J_H^{(1)}(T,a) & = \int_0^t\int_0^\infty t^{H-1/2}(y+a t- \tfrac{t}{y})p(t,y;a,T){\rm d}y{\rm d}t \\
& = \int_0^t t^{H-2}p_2(t;T,a)\int_0^\infty(y^2+aty-t)\exp\left\{-\frac{(y+at)^2}{2t}\right\}{\rm d}y{\rm d}t \\
& = \int_0^t t^{H-1/2}p_2(t;T,a)\int_{a\sqrt{t}}^\infty(z^2-a\sqrt{t}z-1)\exp\left\{-\frac{z^2}{2}\right\}{\rm d}z{\rm d}t,
\end{align*}
where in the last line we substituted $z := \frac{y+at}{\sqrt{t}}$. It can be seen that the definite integral with respect to $z$ equals to $0$ because for every $u\in\R$ we have
\begin{align*}
\int_u^\infty z^2 e^{-\tfrac{z^2}{2}}{\rm d}z = ue^{-\tfrac{u^2}{2}} + \sqrt{\tfrac{\pi}{2}}\erfc\big(\tfrac{u}{\sqrt{2}}\big), \quad \int_u^\infty e^{-\tfrac{z^2}{2}}{\rm d}z = ue^{-\tfrac{u^2}{2}}, \int_u^\infty z^2 e^{-\tfrac{z^2}{2}}{\rm d}z = \sqrt{\tfrac{\pi}{2}}\erfc\big(\tfrac{u}{\sqrt{2}}\big).
\end{align*}
We have established that Eq.~\eqref{eq:J1_zero} holds and therefore it is left to calculate the integral
\begin{align*}
\mathcal J^{(2)}_H(T,a) = \frac{\Gamma(H)}{(H+\tfrac{1}{2})\sqrt{\pi}}\int_0^t t^{H-1}p_2(t;T,a)\int_0^\infty U\left(H-\tfrac{1}{2},\tfrac{1}{2},\tfrac{y^2}{2t}\right)\exp\left\{-\frac{(y+at)^2}{2t}\right\}{\rm d}y{\rm d}t.
\end{align*}
Consider the innermost integral. After substituting $y=z\sqrt{2t}$ we find that
\begin{align*}
\int_0^\infty U\left(H-\tfrac{1}{2},\tfrac{1}{2},\tfrac{y^2}{2t}\right)e^{-\frac{(y+at)^2}{2t}}{\rm d}y & = \sqrt{2t}\int_0^\infty U\left(H-\tfrac{1}{2},\tfrac{1}{2},z^2\right)e^{-(z+a\sqrt{t/2})^2}{\rm d}z.
\end{align*}
%& = \sqrt{\pi}\cdot 2^{H-1}t^{1-H}a^{1-2H}\left( \frac{\gamma(H-\tfrac{1}{2},\tfrac{a^2t}{2})}{\Gamma(H-\tfrac{1}{2})}-\frac{\gamma(H,\tfrac{a^2t}{2})}{\Gamma(H)}\right).
Using Lemma \ref{lemma:U_integral} we obtain
\begin{equation}\label{eq:final_JH_sum}
\mathcal J^{(2)}_H(T,a) = \frac{2^H|a|^{-2H}\Gamma(H)}{\sqrt{2\pi}(H+\tfrac{1}{2})}\bigg(J_1(T,a) + J_2(T,a) - J_3(T,a)\bigg),
\end{equation}
where
\begin{align*}
J_1(T,a) &:= \frac{\sqrt{\pi}2^{-H}|a|^{2H}}{\Gamma(H+\tfrac{1}{2})}\int_0^T p_2(t;T,a)t^{H-1/2}e^{-\tfrac{a^2t}{2}}{\rm d}t\\
J_2(T,a) &:= \frac{\sqrt{\pi}|a|}{\sqrt{2}\Gamma(H+\tfrac{1}{2})}\int_0^T p_2(t;T,a)\gamma(H+\tfrac{1}{2},\tfrac{a^2t}{2}){\rm d}t\\
J_3(T,a) &:= \frac{\sqrt{\pi}a}{\sqrt{2}\Gamma(H)}\int_0^T p_2(t;T,a)\gamma(H,\tfrac{a^2t}{2}){\rm d}t .
\end{align*}
It is left to calculate each of the integrals above. Using substitution $t=tT$ and putting $u:=a\sqrt{T/2}$ we can find that
\begin{align*}
J_1(T,a) & = \frac{|u|^{2H}}{u\Gamma(H+\tfrac{1}{2})}\int_0^1 \left(\frac{u e^{-u^2(1-t)}}{\sqrt{\pi(1-t)}} + u^2\erfc(-u\sqrt{1-t})\right)t^{H-1/2}e^{-u^2t}{\rm d}t \\
J_2(T,a) & = \frac{\sgn(u)}{\Gamma(H+\tfrac{1}{2})}\int_0^1 \left(\frac{ue^{-u^2(1-t)}}{\sqrt{\pi(1-t)}} + u^2\erfc\Big(-u\sqrt{1-t}\Big)\right)\gamma(H+\tfrac{1}{2},u^2t){\rm d}t \\
J_3(T,a) & = \frac{1}{\Gamma(H)}\int_0^1 \left(\frac{ue^{-u^2(1-t)}}{\sqrt{\pi(1-t)}} + u^2\erfc\Big(-u\sqrt{1-t}\Big)\right)\gamma(H,u^2t){\rm d}t
\end{align*}
Let us define:
\begin{align*}
f(t;u):=-\erfc(-u\sqrt{1-t}), \qquad g(t;u,H):=-|u|^{-2H-1}\Gamma(H+\tfrac{1}{2},u^2t).
\end{align*}
Slightly abusing notation, for breviety we write $f(t):=f(t;u)$ and $g(t):=g(t;u,H)$. We then have
\begin{align*}
f'(t) = \frac{ue^{-u^2(1-t)}}{\sqrt{\pi(1-t)}}, \qquad g'(t) = t^{H-1/2}e^{-u^2t}
\end{align*}
and the quantities $J_i(T,a)$, $i\in\{1,2,3\}$ can be expressed as
\begin{align*}
J_1(T,a) & = \frac{|u|^{2H}}{u\Gamma(H+\tfrac{1}{2})}\int_0^1 \Big(f'(t) - u^2f(t)\Big)g'(t){\rm d}t \\
J_2(T,a) & = \frac{1}{\Gamma(H+\tfrac{1}{2})}\int_0^1 \Big(f'(t) - u^2f(t)\Big)\gamma(H+\tfrac{1}{2},u^2t){\rm d}t \\
J_3(T,a) & = \frac{1}{\Gamma(H)}\int_0^1 \Big(f'(t) - u^2f(t)\Big)\gamma(H,u^2t){\rm d}t.
\end{align*}
Before we calculate the values of $J_i(T,a)$ we introduce two useful functions, for $s>0, u\in\R$:
\begin{align*}
h_1(u,s) & := \int_0^1 f'(t)\gamma(s,u^2t){\rm d}t = \sgn(u)\cdot\frac{\Gamma(s)\gamma(s+\tfrac{1}{2},u^2)}{\Gamma(s+\tfrac{1}{2})}, \\
h_2(u,s) & := \int_0^1 (1-t)f'(t)\gamma(s,u^2t){\rm d}t = \sgn(u)\cdot\frac{\Gamma(s)\gamma(s+\tfrac{3}{2},u^2)}{2u^2\Gamma(s+\tfrac{3}{2})}.
\end{align*}
The values of these functions were found by applying relation \eqref{eq:lower_gamma_kummer} and finding the value of the integral using \eqref{key_integral}. Now, integration by parts yields
\begin{align*}
J_1(T,a) & = \frac{|u|^{2H}e^{-u^2}}{\Gamma(H+1)} + \frac{u|u|^{2H}}{\Gamma(H+\tfrac{1}{2})} \cdot \left(f(t)g(t)\bigg|_0^1 - \int_0^1 f'(t)g(t){\rm d}t\right) \\
& = \frac{|u|^{2H}e^{-u^2}}{\Gamma(H+1)} + \frac{\sgn(u)}{\Gamma(H+\tfrac{1}{2})}\Big(\gamma(H+\tfrac{1}{2},u^2) + h_1(u,H+\tfrac{1}{2})\Big) \\
& = \frac{\gamma(H,u^2)}{\Gamma(H)}-\frac{\gamma(H+1,u^2)}{\Gamma(H+1)} + \frac{\sgn(u)}{\Gamma(H+\tfrac{1}{2})}\Big(\gamma(H+\tfrac{1}{2},u^2) + h_1(u,H+\tfrac{1}{2})\Big),
\end{align*}
where in the last line we used the recurrence relation for the incomplete Gamma function in Eq.~\eqref{eq:gamma_recurrence}.
Now, notice that for any $s>0,u\in\R$:
\begin{align*}
\int\gamma(s,u^2t){\rm d}t = t\gamma(s,u^2t)+u^{-2}\Gamma(s+1,u^2t) + C.
\end{align*}
Therefore, integration by parts yields
\begin{align*}
\int_0^1f(t)\gamma(s,u^2t){\rm d}t & = -\gamma(s,u^2)- \int_0^1 tf'(t)\gamma(s,u^2t)+u^{-2}\int_0^1f'(t)\gamma(s+1,u^2t)){\rm d}t \\
& = -\gamma(s,u^2)+ h_2(u,s)-h_1(u,s)+u^{-2}h_1(u,s+1).
\end{align*}
Finally, this gives us
\begin{align*}
J_2(T,a) = \frac{\sgn(u)}{\Gamma(H+\tfrac{1}{2})} & \bigg((1+u^2)h_1(u,H+\tfrac{1}{2})+u^2\gamma(H+\tfrac{1}{2},u^2)-u^2h_2(u,H+\tfrac{1}{2})\\
& \quad -\gamma(H+\tfrac{3}{2})-h_1(u,H+\tfrac{3}{2})\bigg).
\end{align*}
Using analogous methods, we find that
\begin{align*}
J_3(T,a) & = \frac{1}{\Gamma(H)} \bigg((1+u^2)h_1(u,H)+u^2\gamma(H,u^2)-u^2h_2(u,H)-\gamma(H+1,u^2)-h_1(u,H+1)\bigg).
\end{align*}
Through straightforward algebraic manipulations and applying simple recurrence relation \eqref{eq:gamma_recurrence}, we finally obtain
\begin{align*}
J_1(T,a)+J_2(T,a)-J_3(T,a) = \frac{\gamma(H,u^2)}{\Gamma(H)},
\end{align*}
which concludes the proof.
\end{proof}

%---------------------------------------------------------
%---------------------------------------------------------
%---------------------------------------------------------

\section*{Acknowledgements}

The author would like to thank Prof. Krzysztof D\c{e}bicki and Prof. Tomasz Rolski for helpful discussions. Krzysztof Bisewski's research was funded by SNSF Grant 200021-196888.

\medskip
\bibliographystyle{plain}
\bibliography{refs}{}
\end{document}